\documentclass[11pt]{article}

\usepackage{amsthm, amsmath, amssymb, amsfonts, url, booktabs, tikz, setspace, fancyhdr, amsbsy}
\usepackage{fullpage}
\usepackage{hyperref, enumerate}
\usepackage{float}
\usepackage{subcaption}
\usepackage{esint}
\usepackage{verbatim}
\usepackage{multirow}
\usepackage{kotex}
\usepackage{graphicx, wrapfig}
\usepackage{comment}
\usepackage[font=small, labelfont=bf, width=\textwidth]{caption}

\newtheorem{theorem}{Theorem}[section]
\newtheorem{assumption}[theorem]{Assumption}
\newtheorem{remark}[theorem]{Remark}
\newtheorem{proposition}[theorem]{Proposition}
\newtheorem{lemma}[theorem]{Lemma}
\newtheorem{corollary}[theorem]{Corollary}

\theoremstyle{definition}
\newtheorem{definition}[theorem]{Definition}
\newtheorem{example}[theorem]{Example}

\newcommand{\abs}[1]{\left\lvert#1\right\rvert}

\newcommand{\ang}[1]{\left\langle #1 \right\rangle}

\newcommand{\R}{\mathbb{R}}

\newcommand{\dx}{\,\mathrm{d}x}

\newcommand{\dmu}{\,\mathrm{d}\mu}
\newcommand{\dy}{\,\mathrm{d}y}

\newcommand{\NN}{\mathcal{N}}

\newcommand{\tv}{\tilde{v}}
\newcommand{\te}{\tilde{e}}
\newcommand{\ps}{W^{1,p}(\Omega)}
\newcommand{\pb}{W^{1-\frac{1}{p},p}(\partial\Omega)}
\newcommand{\pis}{W^{-1,p'}(\Omega)}

\newcommand{\pa}{\partial}

\newcommand{\calP}{\mathcal{P}}

\newcommand{\sk}[1]{{\color{orange}{#1}}}
\newcommand{\ke}[1]{{\color{blue}{#1}}}

\newcommand{\innerproduct}[2]{\langle #1, #2 \rangle}
\newcommand{\calL}{\mathcal{L}}

\newcommand{\calW}{\mathcal{W}}
\newcommand{\calV}{\mathcal{V}}

\numberwithin{equation}{section}

\newcommand{\dk}[1]{{\color{purple}[DK: #1]}}
\newcommand{\dku}[1]{{\color{purple}#1}}

\DeclareMathOperator{\diver}{div}

\begin{document}

\title{Robust training and rigorous error analysis of physics-informed neural networks for the $p$-Laplace equation}

\author{Kyueon Choi,\thanks{School of Mathematics \& Computing (Mathematics), Yonsei University, Seoul, Republic of Korea. Email: \tt{kyueon92@gmail.ac.kr}}
~Seungchan Ko\thanks{Graduate School of AI for Math, Korea Advanced Institute of Science and Technology, Daejeon, Republic of Korea. Email: \tt{ksm0385@gmail.com}}
~and ~Dohyun Kwon\thanks{School of Mathematics \& Computing (Computational Science \& Engineering), Yonsei University, Seoul, Republic of Korea. 
Email: \tt{dohyunkwon@yonsei.ac.kr}}
\thanks{Center for AI and Natural Sciences, Korea Institute for Advanced Study, Seoul, Republic of Korea.}
}




\date{~}

\maketitle

~\vspace{-1.5cm}

\begin{abstract}
    With the rise of scientific machine learning, physics-informed neural networks (PINNs) have been extensively applied to a wide range of problems. Nevertheless, most theoretical analyses of PINNs remain confined to linear equations, and a substantial gap persists between PINNs and classical numerical analysis for nonlinear problems. To address this issue, we propose a robust training framework for PINNs solving nonlinear partial differential equations, together with a rigorous error analysis. Specifically, for the $p$-Laplace equation, we introduce a novel loss formulation that combines a dual residual loss measured in $W^{-1,p'}$ with a boundary loss measured in a fractional Sobolev norm $W^{1-\frac{1}{p},p}$. This formulation is designed to accommodate the limited regularity of weak solutions and enables us to establish rigorous \textit{a priori} and \textit{a posteriori} error estimates. Moreover, the proposed framework and its analysis are extended to a parametric setting in which the exponent, the source term, and the boundary condition may all vary with the parameters. Finally, we present numerical experiments that substantiate our theoretical findings.
\end{abstract}


\smallskip



\section{Introduction}

In recent years, the intersection of classical scientific computing and modern machine learning has given rise to a new interdisciplinary domain known as scientific machine learning. This emerging field introduces a new framework for solving partial differential equations (PDEs), offering a potential alternative to classical numerical methods. At the forefront of this development are physics-informed neural networks (PINNs) \cite{pinn01, pinn02}, which integrate physical principles directly into deep neural network architectures. By exploiting the universal approximation property of neural networks, PINNs are capable of learning solutions to PDEs while adhering to the underlying physics laws. Owing to their conceptual simplicity and remarkable flexibility, PINNs have rapidly gained attention and have been successfully applied to a wide range of problems in computational science, including biomedical modeling \cite{pinn_bio_1, pinn_bio_2}, fluid mechanics \cite{pinn_fluid_1, pinn_fluid_2, pinn_fluid_3, pinn_fluid_4, NS_comp}, uncertainty quantification \cite{pinn_uq_1, pinn_uq_2, pinn_uq_3}, and meta-material design \cite{pinn_meta_1, pinn_meta_2}. Furthermore, neural networks have proven effective in overcoming the computational challenges of high-dimensional PDEs \cite{pinn_hd_1, pinn_hd_2} and complex geometric domains \cite{PINN_field_1, PINN_domain}. 

Despite these empirical successes, the error analysis for PINNs, unlike that of many classical numerical schemes, has only recently begun to be systematically established, with initial theoretical foundations naturally focusing on linear PDEs. In this context, recent work by Zeinhofer, Masri, and Mardal \cite{zeinhofer2025unified} provides a unified framework with \textit{a priori} and \textit{a posteriori} error estimates in an abstract bilinear-form setting, covering elliptic, parabolic, and hyperbolic equations, as well as related systems. Further theoretical studies in the linear regime include uniform convergence analyses in the sample limit for linear second-order elliptic and parabolic equations using the Schauder approach \cite{shin2020convergence}, alongside error estimates for both continuous and discrete residual minimization formulations \cite{shin2023error}. Additionally, De Ryck and Mishra \cite{de2022error} provided a comprehensive error analysis for PINNs approximating linear Kolmogorov PDEs.



While these contributions have consolidated the mathematical foundations in the linear setting, the extension of such rigorous guarantees to nonlinear PDEs necessitates a fundamentally different theoretical perspective. In this direction, error bounds have been established that relate the total approximation error to the training error, the number of training samples, and the intrinsic stability properties of the underlying PDEs, encompassing 
semi-linear parabolic equations and the incompressible Euler equations 
\cite{mishra2023estimates}. Building upon this framework, subsequent analyses have bounded the total error of PINNs in terms of training and quadrature 
errors, thereby exploiting the stability of PDE solutions to derive convergence for more intricate nonlinear systems, most notably the incompressible Navier--Stokes equations \cite{de2024error}. Nevertheless, in contrast to the relatively mature theory available for linear problems, the rigorous analysis of PINNs for nonlinear PDEs remains largely underdeveloped, with existing results covering only a limited class of equations under restrictive regularity assumptions. Bridging this gap constitutes one of the central contributions of the present work: by carrying out a rigorous theoretical analysis of PINNs for the nonlinear $p$-Laplace equation, we take a further step toward narrowing the divide between physics-informed learning and classical numerical analysis for nonlinear problems.

On the other hand, despite ongoing efforts to develop error estimation frameworks for PINNs, the vast majority of these analyses remain confined to the classical $L^2$ losses. However, the fundamental deficiencies of this standard setup have been recently reported; \cite{DFRPINN, Gazoulis2023} point out that employing the standard $L^2$ loss is not well-defined when solutions exhibit low regularity or the given data is defined in the sense of distributions rather than classical functions. To overcome these limitations, they both advocated a variational approach, with \cite{shin2023error} further establishing a rigorous error estimate for this variational framework. Further applications of the variational approach for PINNs are explored in \cite{VPINN, hpvpinn, RVPINN}, with corresponding error estimates established in \cite{VPINN_apriori, VPINN_aposteriori}.

Furthermore, the use of the $L^2$ boundary loss also introduces some limitations. Specifically, in \cite{shin2023error, zeinhofer2025unified}, their rigorous analysis revealed a fundamental limitation of standard PINNs: enforcing boundary conditions via a conventional $L^2$ penalty restricts the theoretical convergence of the error to weaker Sobolev norms (e.g., bounding the error at most in $H^\frac{1}{2}$ rather than higher-order norms). This mathematical bottleneck highlights the necessity of employing more appropriate functional spaces for trace evaluation; indeed, the utilization of the $H^{\frac{1}{2}}$ norm for boundary formulations can be found in \cite{Gazoulis2023, trpinn}, where \cite{trpinn} provides a practical method for its numerical computation.

In addition to these challenges, the formulation of PINNs remains inherently problem-specific. More precisely, any modification to the parameters defining the PDEs, such as boundary conditions, initial conditions, or variable coefficients, necessitates retraining the neural network from scratch. To overcome this limitation, several studies have proposed parametric extensions of the PINN framework, either by directly incorporating parameters as inputs \cite{paraPINN1, paraPINN2} or by employing meta-learning techniques \cite{paraPINN3, paraPINN4}. These approaches enable rapid solution predictions for varying parameters without the need for extensive retraining, thereby substantially extending the scope and practicality of standard PINNs over classical numerical methods. This capability is particularly highlighted when solving non-linear equations. Specifically, classical numerical methods require iterative methods (e.g., Picard iteration, Newton's method) to compute the solution for each parameter, which results in significant computational overhead for a large number of parameters.

Building upon the aforementioned discussions, in this paper, we propose a novel training methodology for solving nonlinear PDEs with PINNs and provide a rigorous theoretical analysis thereof. To be more specific, we develop a robust training framework for PINNs applied to the $p$-Laplace equation, a prototypical nonlinear elliptic problem. Since weak solutions generally lack the classical regularity required for strong-form $L^2$-residual minimization, such formulations may be ill-posed and can exhibit severe numerical stiffness. To overcome these difficulties, we formulate the interior loss in terms of the dual residual measured in $W^{-1,p'}$, thereby aligning the training objective with the natural weak formulation of the problem, and define the boundary loss in the fractional Sobolev norm $W^{1-\frac{1}{p},p}$, accurately reflecting the trace space inherent to this formulation. Within this framework, we establish rigorous \textit{a priori} as well as \textit{a posteriori} error estimates that certify the accuracy of a trained solution. Moreover, motivated by the recent work of Kaltenbach and Zeinhofer \cite{pPINN}, we further extend the proposed methodology and its error analysis to parameter-dependent $p$-Laplace problems. Before proceeding further, we note that, in a recent work \cite{muga2026residual}, a dual-norm residual minimization for the $p$-Laplacian has been explored within the classical finite element method framework to derive a C\'ea-type lemma and \textit{a posteriori} error estimates under homogeneous Dirichlet boundary conditions. To the best of our knowledge, however, a comparable rigorous foundation for neural network-based approaches has not yet been established, particularly in the presence of general non-homogeneous boundary data and in parametric settings. The main contributions of this paper are summarized as follows.

\begin{itemize}
    \item[1.] \textbf{Robust training via dual and fractional Sobolev norms}: Since weak solutions of the $p$-Laplace equation typically lack higher-order regularity, the standard $L^2$ residual minimization may be ill-posed. To remedy this, we formulate the interior loss in terms of the $W^{-1,p'}$ dual norm of the PDE residual, and measure the boundary residual in the fractional Sobolev norm $W^{1-\frac{1}{p},p}$. This formulation aligns the training objective with the natural functional setting of the problem and yields a mathematically rigorous characterization of both interior and trace errors.
    
    \item[2.] \textbf{Rigorous error estimates}: We carry out a systematic error analysis for the proposed robust PINN formulation. Exploiting the monotonicity of the $p$-Laplace operator, we derive \textit{a priori} error bounds that guarantee the convergence of the proposed method, together with \textit{a posteriori} estimates that directly relate the training loss to the actual solution error. These results ensure that minimizing the proposed loss functional guarantees convergence to the true solution, even in the absence of higher regularity.
    
    \item[3.] \textbf{Extension to parametric problems}: We generalize the proposed framework and its error analysis to parametric $p$-Laplace problems, allowing for variable exponents, source terms, and boundary conditions. Within a single unified training process, the resulting model generalizes across a broad family of $p$-Laplace problems and enables near real-time solution prediction for varying parameters, thereby highlighting a computational advantage over classical numerical methods.
\end{itemize}

The remainder of this paper is organized as follows. Section \ref{sec:prelim} introduces the necessary notations and auxiliary tools required for our theoretical analysis, followed by a brief overview of the $p$-Laplace equation. In Section \ref{sec:cons}, we investigate the specific challenges inherent in $p$-Laplacian problems, providing the mathematical motivation for our approach, and introduce a robust loss functional. Section \ref{sec:err_est} is devoted to the derivation of \textit{a priori} and \textit{a posteriori} error estimates, which are then extended to parametric problems in Section \ref{sec:err_est_para}. Subsequently, Section \ref{sec:num_exp} presents numerical experiments to validate our theoretical findings. Finally, we make a conclusion in Section \ref{sec:concl}.

\section{Preliminary}\label{sec:prelim}

\subsection{Notations and auxiliary results}

This section introduces the necessary notations and auxiliary results, which will be used throughout the paper. The symbol $C$ denotes a generic positive constant that may change from line to line. Let $\Omega \subset \mathbb{R}^d$, $d \in \mathbb{N}$. Then, for $k \in \mathbb{N} \cup \{0\}$ and $p \in [1, \infty)$, we denote the standard Lebesgue and Sobolev spaces by $L^p(\Omega)$ and $W^{k,p}(\Omega)$, respectively. Also, the space $W_0^{k,p}(\Omega) := \overline{C^\infty_0(\Omega)}^{\|\cdot\|_{W^{k,p}(\Omega)}}$ is defined as the subspace of $W^{k,p}(\Omega)$. For $p\ge1$, let $p'$ be the H\"older conjugate exponent satisfying $\frac{1}{p} + \frac{1}{p'} = 1$ and let $W^{-1,p'}(\Omega)$ denote the dual space of $W^{1,p}_0(\Omega)$, equipped with the norm
\begin{align}\label{eq:dualnorm}
    \|f\|_{W^{-1,p'}(\Omega)} := \sup_{\varphi \in W^{1,p}_0(\Omega) \setminus\{0\}} \frac{\langle f,\varphi\rangle_{W^{-1,p'}(\Omega) \times W^{1,p}_0(\Omega)}}{\|\varphi\|_{W^{1,p}(\Omega)}}.
\end{align}
Here, $\innerproduct{\cdot}{\cdot}_{W^{-1,p'}(\Omega) \times W^{1,p}_0(\Omega)}$ represents the duality pairing of $W^{-1,p'}(\Omega)$ and $W^{1,p}_0(\Omega)$. When the context is clear, we omit the subscripts and denote the duality pairing simply by $\langle \cdot, \cdot \rangle$.
Next, we introduce the fractional Sobolev space $W^{s,p}(\Omega)$ for $s \in (0,1)$ and $p \in [1, \infty)$, which is defined as
\begin{equation*}
W^{s,p}(\Omega):=\left\{\, u\in L^{p}(\Omega)\ :\int_{\Omega}\int_{\Omega}\ \frac{|u(x)-u(y)|^p}{|x-y|^{d+sp}}\,{\rm{d}}x\,{\rm{d}}y<\infty\right\},
\end{equation*}
equipped with the norm 
\begin{equation*}
\|u\|_{W^{s,p}(\Omega)}
:=\left(\int_{\Omega}|u|^{p}\dx
\;+\;
\int_{\Omega}\int_{\Omega}\frac{|u(x)-u(y)|^{p}}{|x-y|^{d+sp}}\dx\dy\right)^\frac{1}{p}.
\end{equation*}
To handle the appropriate trace spaces in later analysis, we will consider fractional Sobolev spaces defined on the boundary $\partial\Omega$ using the ($d-1$)-dimensional surface measure. 

We now turn to several auxiliary results which will play crucial roles in our analysis. The trace theorem and the boundary-lifting operator presented below are based on \cite[Theorem 5.5 and Theorem 5.7 in Chapter 2]{necas2011direct}. These results will serve as essential tools for rigorously handling boundary terms in our subsequent error analysis.

\begin{theorem}[Trace theorem]
\label{thm:trace}
    Let $\Omega \subset \mathbb{R}^d$, $d \in \mathbb{N}$ be a bounded Lipschitz domain and $1 < p < \infty$. Then, there exists a bounded linear operator
    \[
    \gamma: W^{1,p}(\Omega) \to \pb,
    \]
    referred to as the trace operator, and a constant $C_p > 0$ depending continuously on $p$, such that 
    \begin{align*}
        \|\gamma u\|_{\pb} \le C_p\,\|u\|_{W^{1,p}(\Omega)}
    \end{align*}
    for any $u \in W^{1,p}(\Omega)$. 
\end{theorem}

\begin{remark}
    In this context,  $\gamma u$ denotes the trace of $u$ on $\partial \Omega$ and we say that $u =g$ on $\partial \Omega$ in the sense of traces if $\gamma u = g$.
\end{remark}

\begin{theorem}[Lifting operator]
\label{thm:blift}
Let $\Omega \subset \mathbb{R}^d$, $d \in \mathbb{N}$ be a bounded Lipschitz domain and $1 < p < \infty$. Then, there exists a bounded linear operator
\[
Z : \pb \to W^{1,p}(\Omega),
\]
referred to as the lifting operator, such that for any $h \in \pb$,
\[
Zh = h \quad \text{on } \partial \Omega.
\]
Furthermore, there exists a constant $C_p > 0$, depending continuously on $p$, such that 
\[
\big\|Zh\big\|_{W^{1,p}(\Omega)} \le C_p \|h\|_{\pb}.
\]
\end{theorem}



We recall the Poincaré inequality, which is fundamental for the analysis of functions with vanishing boundary values. The following result is standard and can be found in \cite[Corollary 9.19]{brezis}.

\begin{lemma}[Poincar\'e inequality]
\label{lem:poi}
Let $\Omega \subset \mathbb{R}^d$, $d \in \mathbb{N}$ be a bounded domain and $1 \le p < \infty$. Then, there exists a constant $C_p>0$, depending continuously on $p$, such that
\begin{align*}
        \|u\|_{L^{p}(\Omega)} \le C_p\,\|\nabla u\|_{L^{p}(\Omega)}
    \end{align*}
for all $u \in W^{1,p}_0(\Omega)$.
\end{lemma}

\subsection{Neural networks}
We now introduce the notation and formal definitions for the feed-forward neural networks. For a depth $L \in \mathbb{N}$ and layer widths $n_0, n_1, \dots, n_L \in \mathbb{N}$, an $L$-layer neural network is represented by a mapping $f^L : \mathbb{R}^{n_0} \to \mathbb{R}^{n_L}$ defined recursively as
\begin{equation}\label{nn_def}
\begin{aligned}
f^1(x) &= W^1 x + b^1, \\
f^\ell(x) &= W^\ell \sigma \big(f^{\ell-1}(x)\big) + b^\ell, \quad \text{for } 2 \le \ell \le L,
\end{aligned}
\end{equation}
where $W^\ell \in \mathbb{R}^{n_\ell \times n_{\ell-1}}$ and $b^\ell \in \mathbb{R}^{n_\ell}$ denote the weight matrix and bias vector of the $\ell$-th layer, respectively, and $\sigma : \mathbb{R} \to \mathbb{R}$ is a nonlinear activation function which is applied componentwise. We collect all trainable parameters into $\theta := \{(W^\ell, b^\ell)\}_{\ell=1}^L$ and define the corresponding parameter space as
\begin{equation}\label{nn_para_set}
\Theta := \left\{ \{{(W^{\ell}, b^{\ell})}\}_{\ell=1}^L : W^{\ell} \in \mathbb{R}^{n_{\ell} \times n_{\ell-1}},\,  b^{\ell} \in \mathbb{R}^{n_{\ell}} \right\}.
\end{equation}
For each $\theta \in \Theta$, the associated neural network function $u_\theta$ is given by the realization of the recursive mapping in \eqref{nn_def}, i.e., $u_\theta(x) := f^L(x)$. Accordingly, the class of feed-forward neural networks is defined as
\begin{equation}\label{nn_class}
\mathcal{N}_\Theta := \left\{ u_\theta : \theta \in \Theta \right\}.
\end{equation}
We note that if the activation function $\sigma$ is smooth, then $\mathcal{N}_\Theta \subset W^{k,p}(\Omega)$ for any $k \in \mathbb{N}$ and $p \in [1, \infty]$.

To characterize the approximation properties of neural networks in relation to solution regularity, we shall employ the following quantitative results of neural network approximation.

\begin{theorem}[Theorem 4.9 in \cite{UAT_quan_1}]\label{UAT_quan_1}
Let $\Omega \subset \mathbb{R}^d$, $d \in \mathbb{N}$ be a bounded Lipschitz domain. Moreover, let $p \in [1,\infty]$ and $k,m\in\mathbb{N} \cup \{0\}$ with $k>m$.
For any $n\in\mathbb{N}$ and any target function $u\in W^{k,p}(\Omega)$, there exists a feed-forward $\tanh$ neural
network $u_{\theta_n}\in W^{m,p}(\Omega)$ with the parameter space of dimension
$\mathcal{O}(n)$, such that the following error estimate holds: for any arbitrarily small $\tau>0$,
\[
\|u - u_{\theta_n}\|_{W^{m,p}(\Omega)}
\;\le\; C_{k,p}
\left(\frac{1}{n}\right)^{\frac{k-m-\tau}{d}}
\; \|u\|_{W^{k,p}(\Omega)},
\]
where $C_{k,p}>0$ depends on $k$ and $p$.
\end{theorem}

\begin{remark}
    Although the above result assumes $\tanh$ activation function, they still hold for many commonly used activation functions (see, e.g., the tables presented in \cite{Barron_1, UAT_quan_1}). 
\end{remark}

\begin{remark}
    If the solution has sufficient regularity, this result may lead to a useful error estimate. However, if the solution has relatively low regularity or if one considers a high-dimensional equation with a large input dimension, the above result shows that the convergence rate may deteriorate. In this case, Barron-type approximation results, which guarantee dimension-independent convergence rates, can be effectively employed (see, e.g., \cite{Barron_1, UAT_quan_2}).
\end{remark}

\subsection{Overview of \texorpdfstring{$p$}{p}-Laplace equation}

We begin our discussion by stating the standard boundary-value problem for the $p$-Laplacian. For \(1<p<\infty\), we consider the following equation:
\begin{equation} \label{p_lap}
\begin{aligned}
    -{\rm{div}}\, (|\nabla u|^{p-2}\nabla u)&=f\quad\text{in }\Omega,\\
    u&=g\quad\text{on }\partial\Omega,
\end{aligned}
\end{equation}
where the domain $\Omega \subset \mathbb{R}^d$, $d \in \mathbb{N}$ is sufficiently regular, and the data \(f\) and \(g\) are chosen as to ensure the existence of a weak solution \(u \in W^{1,p}(\Omega)\). The \(p\)-Laplacian can be understood as a canonical nonlinear generalization of the Laplace operator, with broad relevance in both modeling and analysis \cite{lindqvist2019notes}. On the modeling side, it governs nonlinear diffusion and appears in non-Newtonian fluid mechanics \cite{malek}, nonlinear elasticity and plasticity \cite{Fuchs}, image denoising and inpainting (with total variation flow emerging in the limit \(p \to 1\)) \cite{rudin1992nonlinear}, potential theory via \(p\)-harmonic functions \cite{Heinonen}, and heterogeneous or electrorheological media modeled by variable-exponent \(p(x)\)-Laplacians \cite{Ruzicka2004}. Analytically, the exponent \(p\) controls qualitative features of the operator: for \(p>2\) the operator is degenerate, while for \(1<p<2\) it is singular; moreover, it arises as the Euler--Lagrange equation of the convex functional $\mathcal{E}_p(u) = \frac{1}{p}\int_\Omega |\nabla u|^p\dx$, which ensures coercivity and yields the monotonicity of the associated operator for all \(1<p<\infty\).

These $p$-dependent features render the $p$-Laplacian a particularly 
informative benchmark for PINNs. By varying $p$, one can continuously tune 
the degree of nonlinearity as well as the level of degeneracy or 
singularity, thereby enabling a systematic assessment of the robustness of physics-informed training; a detailed numerical investigation of how the 
training behavior varies with $p$ is presented in Section~\ref{sec:num_exp}. 
Moreover, since the residual depends nonlinearly on $\nabla u$, the equation provides a nonlinear test case of the PINN formulation, while its underlying variational structure naturally suggests suitable loss functionals 
and offers a pathway toward principled \textit{a priori} and \textit{a posteriori} error analysis. Finally, by treating $p$ as an additional parameter or field, the $p$-Laplace family provides a convenient setting for investigating parametric extensions of PINNs. 

For our analysis, we briefly review the analytic framework for the $p$-Laplace equation. Our analysis is fundamentally based on the weak form of the equation, which requires a rigorous definition of weak solutions and their corresponding well-posedness. We begin by recalling the standard notion of weak solutions for \eqref{p_lap}.

\begin{definition}[Weak Solution] \label{def:weaksol}
Let $\Omega \subset \mathbb{R}^d$ be a bounded Lipschitz domain. Given $f \in W^{-1, p'}(\Omega)$ and $g \in \pb $ with $1<p<\infty$, we say that $u \in W^{1,p}(\Omega)$ is a weak solution to \eqref{p_lap} if for every test function $\phi \in W_0^{1,p}(\Omega)$, there holds
    \begin{equation} \label{weaksol}
        \int_{\Omega} |\nabla u|^{p-2} \nabla u \cdot \nabla \phi  \dx = 
        \langle f,\phi\rangle
    \end{equation}
and $u=g$ on $\partial \Omega$ in the sense of the traces of functions in $W^{1,p}(\Omega)$.
\end{definition}

Based on Theorems~\ref{thm:trace} and \ref{thm:blift}, the well-posedness of a solution of \eqref{p_lap} with suitable $f$ and $g$ is obtained. The proof can be found, for instance, in \cite[Theorem 2.16]{lindqvist2019notes}.

\begin{theorem}[Existence and uniqueness] \label{thm:exist}
    Let $\Omega \subset \mathbb{R}^d$ be a bounded Lipschitz domain. Assume that $f \in W^{-1,p'}(\Omega)$ and $g \in \pb$ with $1<p<\infty$. Then, there exists a unique weak solution of \eqref{p_lap}.
\end{theorem}

We then derive the following standard $W^{1,p}$-estimate for the weak solution, which will be frequently utilized in the later analysis.
\begin{proposition}\label{lp_est}
     Let $u \in W^{1,p}(\Omega)$ be a weak solution of \eqref{p_lap}. Then, there exists a constant $C_p>0$, depending continuously on $p$, such that
    \[
        \|\nabla u\|_{L^p(\Omega)}^p\leq C_p  \left(\|f\|^{p'}_{W^{-1,p'}(\Omega)} + \|g\|_{\pb}^p\right).
    \]
\end{proposition}

\begin{proof}
According to Theorem~\ref{thm:blift}, there exists a lifting operator $Z:\pb \rightarrow W^{1,p}(\Omega)$ such that $G := Zg \in W^{1,p}(\Omega)$ satisfies $G = g$ on $\partial \Omega$ and 
\[
\|G\|_{W^{1,p}(\Omega)} \le C_p \|g\|_{\pb}.
\]
Setting $\phi := u - G$, we have $\phi \in W^{1,p}_0(\Omega)$ and it follows from Definition~\ref{def:weaksol} that
\begin{equation} \label{W1p est of u}
\int_\Omega |\nabla u| ^{p} \dx = \int_\Omega |\nabla u|^{p-2}\nabla u \cdot \nabla G \dx + \langle f,u-G\rangle
=: {\rm A_1} + {\rm A_2}.
\end{equation}
First, by applying H\"older's and Young's inequalities, we obtain for a sufficiently small $\epsilon > 0$ that
\[
{\rm A_1} \le \epsilon\|\nabla u\|_{L^p(\Omega)}^p + C_{\epsilon} \big\|\nabla G\big\|_{L^p(\Omega)}^p \le \epsilon \|\nabla u\|_{L^p(\Omega)}^p + C_{\epsilon,p} \|g\|_{\pb}^p.
\]
For the second term $\rm A_2$, we employ a duality estimate, the Poincar\'e inequality, and Young's inequality to obtain for a sufficiently small $\epsilon > 0$ that
\begin{align*}
    {\rm A_2} &\le \|f\|_{W^{-1,p'}(\Omega)} \|u - G\|_{W^{1,p}(\Omega)} \\
    & \le C_p \|f\|_{W^{-1,p'}(\Omega)} \big\|\nabla(u - G)\big\|_{L^p(\Omega)} \\
    &\le C_p \|f\|_{W^{-1,p'}(\Omega)}\big( \|\nabla u\|_{L^p(\Omega)} + \big\|\nabla G\big\|_{L^p(\Omega)}\big) \\
    &\le \epsilon \|\nabla u\|_{L^p(\Omega)}^p + C_{\epsilon,p} \|f\|_{W^{-1,p'}(\Omega)}^{p'} + C_p\|f\|_{W^{-1,p'}(\Omega)}^{p'} + C_p\|g\|_{\pb}^p.
\end{align*}
Finally, substituting the estimates for $\rm A_1$ and $\rm A_2$ into \eqref{W1p est of u} yields the desired estimate.
\end{proof}

\section{Robust training of PINNs for the \texorpdfstring{$p$}{p}-Laplace equation}\label{sec:cons}

In this section, we introduce a novel approach to defining the PINN loss for the $p$-Laplacian problem. As motivating examples, we first discuss the difficulties that arise when solving the $p$-Laplace equation using a standard PINN. Motivated by these challenges, we propose a loss functional based on the dual norm of the PDE residual and the fractional Sobolev norm for the boundary residual, thereby providing a mathematically rigorous framework for the $p$-Laplacian problem.

\subsection{{Standard PINNs and some challenges}}
In the standard PINN framework, the loss functional is defined as the $L^2$-norm of the residual. More specifically, the population loss functional for $p$-Laplace problem \eqref{p_lap} is given as 
\[
\mathcal{L}_{\text{strong}}(v) = \delta_r \big\|-\diver\,(\abs{\nabla v}^{p-2}\nabla v)-f\big\|_{L^2(\Omega)} + \delta_b \big\| \gamma v-g\big\|_{L^2(\partial\Omega)},
\]
where $\delta_r$ and $\delta_b$ are weight parameters and $\gamma$ is the trace operator defined in Theorem \ref{thm:trace}. Then, we aim to solve the minimization problem
\begin{align}
\label{p:strong}
    \inf_{v \in \mathcal{A}_{\text{strong}}} \mathcal{L}_{\text{strong}}(v),
\end{align}
where
\begin{equation*}
\mathcal{A}_{\text{strong}}
= \left\{v \in C^2(\Omega) : \diver\,(\abs{\nabla v}^{p-2}\nabla v) - f \in L^2(\Omega), \, \gamma v-g \in L^2(\partial \Omega) \right\}.
\end{equation*}
However, the standard PINN loss formulation exhibits several theoretical and computational drawbacks for certain classes of problems, making it inadequate for handling more general cases. For instance, even for the linear case, it was reported in \cite{DFRPINN, Gazoulis2023} that the classical strong-form PINN loss functional becomes ill-defined when the solutions lack $H^2$ regularity or the forcing term $f$ belongs to $H^{-1}(\Omega)$. Furthermore, even when the loss functional is mathematically well-defined, minimizing the discrete loss does not necessarily guarantee convergence in the continuous $L^2$ norm. Specifically, if the PDE residual takes the form of a high-frequency sine function, the discrete training loss can completely vanish while the true error blows up \cite[Example 4.1]{shin2023error}. 





We first consider a point-source problem in which \(f\) is a Dirac mass. 
While the linear case (\(p=2\)) was studied in \cite{DFRPINN}, the \(p\)-Laplace equation poses analogous challenges. 
Since \(f\) is a distribution rather than a square-integrable function, i.e., \(f \notin L^2(\Omega)\), the standard PINN formulation is not directly applicable, and a framework suited to such low-regularity data is required. The following example serves as a fundamental benchmark for this purpose.
\begin{example} \label{ex:dirac}
Let $\Omega = (-1,1)$ and fix $1<p<\infty$. The function
\[
u^*(x) = \left( \frac{1}{2} \right)^{\frac{1}{p-1}} (1 - |x|)
\]
is the weak solution of
\begin{equation*} 
-\frac{\rm{d}}{{\rm d}x}\Big(\abs{u'}^{p-2}u'\Big) = \delta_0 \quad \text{in } \Omega,
\end{equation*}
subject to the boundary condition $u(1)=u(-1)=0$, where $\delta_0$ denotes the Dirac mass centered at zero.
\end{example}

Moreover, the challenges associated with the standard PINNs extend beyond the case of singular distributions. Even under the assumption of a highly smooth $f$, the nonlinear and degenerate nature of the $p$-Laplace operator introduces severe stiffness in the loss landscape. This is illustrated by the following specific setting.

\begin{example}\label{setup:counter_example}
Let $\Omega=(-1,1)$, $1<p<\infty$, and consider the problem:
\begin{equation}\label{eq:1d_pde_general}
\begin{cases}
    -\frac{\rm{d}}{{\rm d}x}\Big(\abs{u'}^{p-2}u'\Big) = -1 & \text{in } \Omega,\\
    u(-1)=u(1) = \frac{p-1}{p} & \text{on } \pa\Omega.
\end{cases}
\end{equation}
Then the exact solution is given by:
\begin{equation*}\label{eq:u_star_general}
u^*(x)=\frac{p-1}{p}\abs{x}^{\frac{p}{p-1}}.
\end{equation*}
Indeed, the derivatives can be computed as
\[
(u^*)'(x)=\mathrm{sgn}(x)\abs{x}^{\frac{1}{p-1}},
\quad
(u^*)''(x)=\frac{1}{p-1}\abs{x}^{\frac{2-p}{p-1}} \quad \text{for } x \ne 0,
\]
and it follows that
\begin{equation*}\label{eq:cancellation_general}
-\frac{\rm{d}}{{\rm d}x}\Big(\abs{(u^*)'}^{p-2}(u^*)'\Big)=-(p-1)\abs{(u^*)'(x)}^{p-2}(u^*)''(x) = -1.
\end{equation*}
\end{example}

Although the exact solution $u^*$ in the above example satisfies the PDE a.e., utilizing the standard PINN loss formulation remains problematic for both $p>2$ and $1<p<2$. Specifically, for $p>2$, the singularity of $(u^*)''$ at $x=0$ induces instability in the neural network's approximation capability, making it difficult for the model to learn the solution profile accurately. On the other hand, for $1<p<2$, the singularity of $|(u^*)'|^{p-2}$ at $x=0$ triggers numerical instability during the discrete loss computation. The following theorem provides a \emph{stiffness/trade-off} characterization for $p>2$, which analytically demonstrates the potential training instability inherent in Example \ref{setup:counter_example}. We shall denote the PDE residual of \eqref{eq:1d_pde_general} by
\[
R(u)(x):=(p-1)\abs{u'(x)}^{p-2}u''(x)-1.
\]
\begin{theorem}
\label{thm:counter}
    Let $\Omega = (-1,1)$, $2<p<\infty$, and consider the problem \eqref{eq:1d_pde_general}. Suppose that there exists a sequence of even functions $\{v_n\}_{n \in \mathbb{N}}\subset C^2(\overline{\Omega})$ satisfying $v_n(0)=0$ for all $n \in \mathbb{N}$ and 
$\big\|R(v_n)\big\|_{L^2(-1,1)}\to 0$. Then, $\|v_n''\|_{L^\infty(-1,1)}\to\infty$.
\end{theorem}

The conclusion of Theorem~\ref{thm:counter} is that, in order to make the strong residual small in the $L^2$ norm, a smooth model approximation $v$ might develop increasingly large curvature near the singularity. Indeed, it captures the core difficulty of training PINNs in this regime: the optimizer is forced to produce sharper curvature to minimize the loss, leading to severe stiffness. The proof follows directly from the proposition below.

\begin{proposition}
Let $\Omega = (-1,1)$, $2<p<\infty$, and consider the problem \eqref{eq:1d_pde_general}. Let $v\in C^2(\overline{\Omega})$ be an even function satisfying $v(0)=0$. Then there exists a constant $C_p>0$, depending only on $p$, such that
\begin{equation}\label{eq:tradeoff_general_p}
\big\|R(v)\big\|_{L^2(-1,1)} \, \ge \, C_p\, (1+M)^{-\frac{p-1}{2(p-2)}},
\end{equation}
where $M:=\|v''\|_{L^\infty(-1,1)}$.
\end{proposition}

\begin{proof}
Since $v$ is even and differentiable, $v'(0)=0$. For any $x\in(0,1)$, the mean value theorem implies $|v'(x)| \le Mx$, given that $|v''(x)|\le M$.
Consequently, the flux derivative term can be bounded as
\[
(p-1)\abs{v'(x)}^{p-2}|v''(x)|
\le (p-1)(Mx)^{p-2}M
= (p-1)M^{p-1}x^{p-2}.
\]
Let $\delta:=\left(\frac{1}{2(p-1)(1+M)^{p-1}}\right)^{1/(p-2)}$. Since $p>2$, we have $0<\delta<1$. Then, for all $x\in(0,\delta)$, we have
$(p-1)\abs{v'(x)}^{p-2}|v''(x)|\le \frac12$. This implies that the residual is bounded away from zero in this interval:
\[
|R(v)(x)| = \left| (p-1)\abs{v'(x)}^{p-2}v''(x) - 1 \right| \ge 1 - \frac12 = \frac12.
\]
Integrating the squared residual over $(0, \delta)$ yields
\[
\big\|R(v)\big\|_{L^2(-1,1)}^2 \ge \int_0^\delta \left(\frac12\right)^2 \dx = \frac{\delta}{4}.
\]
Taking the square root proves the desired inequality \eqref{eq:tradeoff_general_p}.
\end{proof}
These observations indicate that the conventional $L^2$-based framework for training PINNs may render the problem intrinsically ill-conditioned and severely hinder the training process. This, in turn, underscores the need for a general and robust formulation that remains stable across the full range of exponents $1<p<\infty$.

Beyond the interior residual, the standard $L^2$ boundary loss also fails to 
sufficiently enforce the required regularity of the solution. More precisely, 
rigorous analyses in \cite{shin2023error, zeinhofer2025unified} show that 
enforcing boundary conditions via a conventional $L^2$ penalty strictly 
restricts the convergence of the interior error to the $H^{\frac{1}{2}}$ norm. For instance, for the Poisson equation, the solution $u$ satisfies the estimate
\[
\|u\|_{H^s(\Omega)} \le C \left( \|\Delta u\|_{L^2(\Omega)} + \|u\|_{L^2(\pa\Omega)} \right)
\]
if and only if $s \le \frac{1}{2}$. To address this limitation in the context of linear elliptic PDEs, \cite{Gazoulis2023, trpinn} suggest utilizing the $H^\frac{1}{2}$ norm for the boundary formulation, which is natural for an $H^1$ solution by the trace theorem. Specifically, for a linear elliptic operator $L$, the solution $u$ satisfies 
\[
\|u\|_{H^1(\Omega)} \le C \left( \|Lu\|_{L^2(\Omega)} + \|u\|_{H^\frac{1}{2}(\pa\Omega)} \right).
\]
Consequently, for the $p$-Laplace problem under consideration, it is natural to employ the norm of $\pb$ for the boundary condition, which constitutes the correct trace space for $W^{1,p}(\Omega)$.

\subsection{Proposed framework}
\label{sec:cons2}

To overcome the issues discussed in the previous subsection, we shall now propose a novel variational approach based on the dual norm and the fractional Sobolev norm. Before proceeding further, for $f \in W^{-1,p'}(\Omega)$ and $g \in \pb$, we define the PDE and boundary residuals for $v \in W^{1,p}(\Omega)$ as
\begin{equation}\label{PDE_res}
    R(v):=-{\rm{div}}\,(|\nabla v|^{p-2}\nabla v)-f\in W^{-1,p'}(\Omega)
\end{equation}
and
\begin{equation}\label{bdry_res}
    B(v):=\gamma v-g\in \pb,
\end{equation}
where $\gamma$ is the trace operator defined in Theorem \ref{thm:trace}. For the linear functional $R(v)$ on $W^{1,p}_0(\Omega)$, we define its duality pairing as
\begin{equation}\label{dual_pair}
    \langle R(v), \varphi \rangle_{\rm{res}} := \int_\Omega |\nabla v|^{p-2}\nabla v \cdot\nabla \varphi \dx - \langle f,\varphi\rangle.
\end{equation}
Equipped with these settings, we propose to consider the following minimization problem over the admissible space $\mathcal{A} = W^{1,p}(\Omega)$:
\begin{align}
\label{p:ideal}
    \inf_{v \in \mathcal{A}} \mathcal{L}(v),
\end{align}
where the loss functional is defined using the dual norm and the fractional Sobolev norm:
\begin{align}
\label{def:loss}
    \mathcal{L}(v) := \delta_r\big\|R(v)\big\|_{W^{-1,p'}(\Omega)}^\alpha + \delta_b\big\|B(v)\big\|_{\pb},
\end{align}
where $\alpha = \max\{1,\frac{1}{p-1}\}$, and $\delta_r, \delta_b$ are weight parameters. The exponent $\alpha$ is motivated by the homogeneity of the $p$-Laplace operator. Since $\xi\mapsto|\xi|^{p-2}\xi$ is $(p-1)$-homogeneous, an error of size $t$ in the solution induces a residual of size $t^{p-1}$. For $1<p<2$, the map $t\mapsto t^{p-1}$ is concave with an infinite slope at the origin, so a loss built directly on the residual has a gradient of order $|t|^{p-2}$, which blows up as $t\to 0$, making the optimizer prone to gradient explosion precisely where training should settle down. Raising the residual to the power $\alpha$ turns this into $|t|^{\alpha(p-1)}$, and requiring $\alpha(p-1)\ge 1$ keeps the gradient bounded, which yields $\alpha\ge\frac{1}{p-1}$. The choice $\alpha=\max\{1,\frac{1}{p-1}\}$ thus enforces this correction for $1<p<2$ while retaining $\alpha=1$ for $p\ge 2$, where no rescaling is needed; as a by-product, both loss components become of the same order in the error near convergence. Note that this functional setting is natural and mathematically sound, as it directly reflects the well-posedness framework of the problem, in which $f \in W^{-1,p'}(\Omega)$ and the boundary values are characterized by the trace space $\pb$ for any $v \in W^{1,p}(\Omega)$.

This variational treatment of the PDE residual, known as the variational PINNs (VPINNs) framework, originated in \cite{VPINN}, where neural networks serve as the trial space and finite-dimensional spaces as the test space, 
improving both accuracy and computational efficiency for certain low-regularity problems. The framework was subsequently extended in 
\cite{hpvpinn} to incorporate $hp$-refinement through domain decomposition of the test space, and rigorous error estimates for this approach were established in \cite{shin2023error}. In a further development, the authors of \cite{RVPINN} proposed robust VPINNs, in which the loss functional is constructed from the Riesz representative of the weak residual. In a 
related direction, \cite{DFRPINN} adopted the $H^{-1}$ norm as a loss functional and developed the deep Fourier residual method for its 
efficient computation.

We now formally establish the consistency of our approach for the $p$-Laplace problem in the following proposition.

\begin{proposition}
    Let $u^* \in W^{1,p}(\Omega)$ be the unique weak solution of \eqref{p_lap}. For any $v \in W^{1,p}(\Omega)$, the weak loss satisfies $\mathcal{L}(v) = 0$ if and only if $v = u^*$.
\end{proposition}

\begin{proof}
    Since norms are non-negative, $\mathcal{L}(v) = 0$ holds if and only if $\big\|R(v)\big\|_{W^{-1,p'}(\Omega)} = 0$ and $\big\|B(v)\big\|_{\pb} = 0$. On one hand, $\big\|R(v)\big\|_{W^{-1,p'}(\Omega)} = 0$ is equivalent to $\langle R(v), \varphi \rangle_{\rm{res}} = 0$ for all $\varphi \in W^{1,p}_0(\Omega)$, which recovers the exact weak formulation:
    \[
        \int_\Omega |\nabla v|^{p-2}\nabla v \cdot \nabla \varphi  \dx = \langle f, \varphi \rangle, \quad \forall \varphi \in W^{1,p}_0(\Omega).
    \]
    On the other hand, $\big\|B(v)\big\|_{\pb} = 0$ guarantees that $v = g$ on $\partial\Omega$ in the trace sense. Consequently, $\mathcal{L}(v) = 0$ if and only if $v$ is a weak solution of \eqref{p_lap}. By the uniqueness of the weak solution $u^*$, we immediately conclude that $v = u^*$.
\end{proof}

\begin{remark}
Furthermore, our analysis fundamentally aligns with the VPINN framework. This connection is mathematically justified by the fact that the vanishing of the corresponding dual norm is equivalent to the vanishing of the weak residual for all test functions; explicitly, we have
\begin{equation} \label{eqiv_vpinn}
\big\|R(v)\big\|_{W^{-1,p'}(\Omega)}=0\quad \text{if and only if} \quad \langle R(v), \varphi \rangle_{\rm{res}} = 0, \quad \forall \varphi \in W^{1,p}_0(\Omega) .
\end{equation}
We will use this observation later in the experimental section to compute the dual norm in our proposed method. See also Remark~\ref{VPINN compute} for the precise definition of our empirical residual loss.
\end{remark}

\section{Error estimates for approximate solutions}\label{sec:err_est}

In this section, we derive {\textit{a priori}} and {\textit{a posteriori}} error estimates (Theorem \ref{thm:posteriori}, Theorem \ref{thm:priori}, Corollary \ref{conv rate}) for the PINN solution solving the $p$-Laplace problem \eqref{p_lap}. To establish a consistent framework for the remainder of this section, we assume that $\Omega \subset \mathbb{R}^d$, $d \in \mathbb{N}$ is a bounded Lipschitz domain, and let the forcing term $f$ and the boundary condition $g$ satisfy
\begin{equation} \label{basic setup}
f\in W^{-1,p'}(\Omega) \quad \text{and} \quad g \in W^{1-\frac{1}{p},p}(\pa \Omega) \quad \text{with } 1<p<\infty.
\end{equation}

We begin with the following lemma, which summarizes fundamental properties of the $p$-Laplacian operator, which can be found in several standard references, including \cite{pL_1, pL_2}.
\begin{lemma}\label{alg_lem}
    For any $1<p<\infty$ and $\xi$, $\eta\in\R^d$ with $\xi\neq\eta$, there exists a constant $C_p$, depending continuously on $p$, such that 
    \begin{align*}
        \left||\xi|^{p-2}\xi-|\eta|^{p-2}\eta\right|&\le C_p \left(|\xi|+|\eta|\right)^{p-2}|\xi-\eta|,\\
        |\xi-\eta|^2\left(|\xi|+|\eta|\right)^{p-2}&\le C_p \left(|\xi|^{p-2}\xi-|\eta|^{p-2}\eta\right)\cdot\left(\xi-\eta\right).
    \end{align*}
    In particular, if $p>2$ we have that
    \[
    |\xi-\eta|^p\le C_p \left(|\xi|^{p-2}\xi-|\eta|^{p-2}\eta\right) \cdot \left(\xi-\eta\right),
    \]
    while, if $1<p<2$, it follows that
    \[
    \left||\xi|^{p-2}\xi-|\eta|^{p-2}\eta\right|\le C_p|\xi-\eta|^{p-1}.
    \]
\end{lemma}

We will also use the following auxiliary result, which addresses the singular case of $1<p<2$ for the $p$-Laplacian in our subsequent error analysis.
\begin{lemma}\label{aux_ineq_1}
    Let $a\ge 0$, $b\ge 0$ with $a,b\in L^p(\Omega)$ and $1<p<2$. Then, there holds
\[
\int_\Omega a^{p-2}b^2\dx
\ge
\frac{\|b\|_{L^p(\Omega)}^{2}}{\|a\|_{L^p(\Omega)}^{2-p}}.
\]
\end{lemma}
\begin{proof}  
We may write $b^p=\big(a^{p-2}b^2\big)^{\frac{p}{2}}\,a^{\frac{(2-p)p}{2}}$.
Integrating this identity over $\Omega$ and applying H\"older's inequality with exponents $\frac{2}{p}$ and $\frac{2}{2-p}$ yield
\[
\int_\Omega b^p\dx
=
\int_\Omega \big(a^{p-2}b^2\big)^{\frac{p}{2}}\,a^{\frac{(2-p)p}{2}}\dx
\le
\bigg(\int_\Omega a^{p-2}b^2\dx\bigg)^{\frac{p}{2}}
\bigg(\int_\Omega a^p\dx\bigg)^{\frac{2-p}{2}}.
\]
By taking both sides to the power $\frac{2}{p}$, we have
\[
\|b\|_{L^p(\Omega)}^2
\le
\bigg(\int_\Omega a^{p-2}b^2 \dx \bigg)\|a\|_{L^p(\Omega)}^{\,2-p},
\]
which is the desired inequality.
\end{proof}


Before establishing our main theoretical bounds, we introduce the neural network ansatz class $\mathcal{N}_{\Theta} \subset W^{1,p}(\Omega)$, consisting of feed-forward neural networks associated with a parameter space $\Theta$. Throughout the analysis, our primary interest lies in the regime where the training has sufficiently progressed. Since the training of PINNs explicitly minimizes both the PDE and boundary residuals, it is natural to expect that these quantities, together with the solution error, eventually fall below a moderate threshold in this regime. We therefore postulate the following conditions on the neural network 
$v = v_\theta \in \mathcal{N}_\Theta$.

\begin{assumption} \label{small assumption}
    The neural network $v=v_\theta \in \mathcal{N}_\Theta$ satisfies the following bounds:
    \[
    \big\|R(v)\big\|_{\pis}\le 1, \quad \big\|B(v)\big\|_{W^{1-\frac{1}{p},p}(\pa\Omega)}\le 1, \quad \|v-u^*\|_{\ps} \le 1,
    \]
    where $R(v)$ and $B(v)$ are given in \eqref{PDE_res} and \eqref{bdry_res}, respectively.
\end{assumption}

In order to derive the error estimates, we first show that the neural network approximation error is bounded above by the PDE and boundary residuals.

\begin{proposition} \label{prop:posteriori}
    Assume that \eqref{basic setup} holds, and let $u^* \in W^{1,p}(\Omega)$ be a weak solution of \eqref{p_lap}. Then, for the neural network $v=v_\theta\in\NN_{\Theta}$ satisfying Assumption \ref{small assumption}, we have the following:
    \begin{itemize}
        \item [(1)] For $p>2$, there exists a constant $C_p>0$, depending continuously on $p$, such that 
    \[
    \|v-u^*\|_{\ps} \le C_p \left(\big\|R(v)\big\|^{\frac{1}{p-1}}_{W^{-1,p'}(\Omega)} +\big\|B(v)\big\|_{\pb}^\frac{1}{p-1}\right).
    \]
        \item [(2)] For $1<p<2$, there exists a constant $C_p>0$, depending continuously on $p$, such that
    \[
    \|v-u^*\|_{\ps}
    \le C_p \left(\big\|R(v)\big\|_{W^{-1,p'}(\Omega)}+\big\|B(v)\big\|_{\pb}^{p-1}\right).
    \]
    \end{itemize} 
\end{proposition}

\begin{proof}
From Theorem~\ref{thm:blift},
for the given boundary residual $B(v)\in \pb$, we can find the extension $w:=ZB(v)\in W^{1,p}(\Omega)$ satisfying
\begin{equation} \label{posteriori 21}
    \|w\|_{W^{1,p}(\Omega)} \le C_p \big\|B(v)\big\|_{\pb}.
\end{equation}
We then set $\tilde{v}:=v-w$ so that $\tilde{v} \in W^{1,p}(\Omega)$ and $\gamma\tilde{v}=g$. We first measure the error $\tilde{e}:=\tilde{v}-u^*\in W^{1,p}_0(\Omega)$. By the definition of the duality pairing $\ang{\cdot,\cdot}_{\rm res}$ in \eqref{dual_pair} and Definition~\ref{def:weaksol}, we see that 
\begin{equation}\label{pf_ini}
\begin{aligned}
    {\langle R(\tilde{v}),\tilde{e}\rangle_{\rm{res}}}
    &=\int_{\Omega} |\nabla \tilde{v}|^{p-2} \nabla \tilde{v} \cdot \nabla \tilde{e}  \dx - {\langle f , \te \rangle} \\
    &=\int_{\Omega}\left(|\nabla\tilde{v}|^{p-2}\nabla\tilde{v}
    -|\nabla u^*|^{p-2}\nabla u^*\right)  
    \cdot(\nabla\te) \dx.
\end{aligned}
\end{equation}

(Case $p > 2$) By Lemma \ref{alg_lem}, we have, by the Poincar\'e inequality, that 
\[
\|\te\|^p_{\ps}\le C_p \,{\langle R(\tv),\te\rangle_{\rm{res}}},
\]
which together with the triangular inequality leads us to the estimate 
\begin{equation} \label{posteriori 11}
\|\te\|^{p-1}_{\ps}\le C_p \big\|R(\tv)\big\|_{W^{-1,p'}(\Omega)} \le C_p \big\|R(v)\big\|_{W^{-1,p'}(\Omega)} + C_p \big\|R(\tv) - R(v)\big\|_{W^{-1,p'}(\Omega)}.
\end{equation}
By applying Lemma \ref{alg_lem} and H\"older's inequality, we obtain
\begin{equation} \label{posteriori 12}
\begin{aligned}
    \big\|R(\tv)-R(v)\big\|_{W^{-1,p'}(\Omega)}
    &=\sup_{\varphi\in W^{1,p}_0(\Omega)\setminus\{0\}}\frac{\langle R(\tv)-R(v),\varphi\rangle_{\rm res}}{\|\varphi\|_{W^{1,p}(\Omega)}}\\
    &\le \big\||\nabla\tv|^{p-2}\nabla\tv-|\nabla v|^{p-2}\nabla v\big\|_{L^{p'}(\Omega)}\\
    &\le C_p \big\|\left(|\nabla\tv|+|\nabla v|\right)^{p-2}|\nabla\tv-\nabla v|\big\|_{L^{p'}(\Omega)}\\
    &\le C_p \big\| |\nabla\tv|+|\nabla v|\big\|^{p-2}_{L^p(\Omega)}\|\tv -v\|_{\ps}.
\end{aligned}
\end{equation}
Hence, \eqref{posteriori 11} combined with \eqref{posteriori 12} and \eqref{posteriori 21} yields
\begin{align*}
    \|\te\|^{p-1}_{\ps}
    &\le C_p\big\|R(v)\big\|_{W^{-1,p'}(\Omega)}+ C_p\big\||\nabla\tv|+|\nabla v|\big\|^{p-2}_{L^p(\Omega)}\big\|B(v)\big\|_{\pb}\\
    &\le C_p\big\|R(v)\big\|_{W^{-1,p'}(\Omega)}+ C_p\left(\|\nabla v\|^{p-2}_{L^p(\Omega)}+\|\nabla w\|^{p-2}_{L^p(\Omega)}\right)\big\|B(v)\big\|_{\pb}\\
    &\le C_p \big\|R(v)\big\|_{W^{-1,p'}(\Omega)}+ C_p\|\nabla v\|^{p-2}_{L^p(\Omega)}\big\|B(v)\big\|_{\pb}+ C_p\big\|B(v)\big\|_{\pb}^{p-1}.
\end{align*}
The second term on the right-hand side can be estimated as
\begin{align*}
\|\nabla v\|^{p-2}_{L^p(\Omega)}\big\|B(v)\big\|_{\pb} &= \|\nabla v - \nabla u^* + \nabla u^*\|_{L^p(\Omega)}^{p-2}\big\|B(v)\big\|_{\pb} \\
&\le C_p \left(\| \nabla v - \nabla u^*\|_{L^p(\Omega)}^{p-2} + \|\nabla u^*\|_{L^p(\Omega)}^{p-2}\right) \big\|B(v)\big\|_{\pb}\\
&\le C_p \| \nabla v - \nabla u^*\|_{L^p(\Omega)}^{p-2}\big\|B(v)\big\|_{\pb} + C_p\big\|B(v)\big\|_{\pb},
\end{align*}
where we have used Proposition \ref{lp_est}. Moreover, by Young's inequality, we have for sufficiently small $\epsilon>0$ that
\[
\|\nabla v - \nabla u^*\|_{L^p(\Omega)}^{p-2} \big\|B(v) \big\|_{\pb} \le \epsilon \|v- u^*\|_{W^{1,p}(\Omega)}^{p-1} + C_\epsilon \big\|B(v)\big\|_{\pb}^{p-1}.
\]
Consequently, applying Young's inequality with a sufficiently small $\epsilon > 0$, we obtain
\begin{align*}
    \|v-u^*\|_{\ps}^{p-1}
    &\le C_p \|\te\|_{\ps}^{p-1} + C_p\|w\|_{\ps}^{p-1} \\
    &\le C_p \big\|R(v)\big\|_{W^{-1,p'}(\Omega)}+\epsilon \|v- u^*\|_{W^{1,p}(\Omega)}^{p-1} \\
    &\hspace{4mm}+C_{\epsilon,p} \big\|B(v)\big\|_{\pb}^{p-1} + C_p \big\|B(v)\big\|_{\pb}.
\end{align*}
Under the assumption that $\big\|B(v)\big\|_{\pb} \le 1$, we arrive at the desired estimate.

(Case $1<p<2$) It follows from Lemma \ref{alg_lem} and \eqref{pf_ini} that
\[
\int_{\Omega}|\nabla\te|^2\big(|\nabla\tv|+|\nabla u^*|\big)^{p-2} \dx\le C_p \, \langle R(\tv),\te\rangle_{\rm{res}}.
\]
By Lemma \ref{aux_ineq_1} and the duality estimate, we see that
\begin{align*}
    \|\nabla\te\|_{L^p(\Omega)}
    &\le C_p\big\|R(\tv)\big\|_{W^{-1,p'}(\Omega)}\big\||\nabla\tv|+|\nabla u^*|\big\|^{2-p}_{L^p(\Omega)}\\
    &\le C_p\big\|R(\tv)\big\|_{W^{-1,p'}(\Omega)}
    \Big(\|\nabla\te\|^{2-p}_{L^p(\Omega)}+\|\nabla u^*\|^{2-p}_{L^p(\Omega)}\Big).
\end{align*}
If $\|\nabla\te\|_{L^p(\Omega)}\ge\|\nabla u^*\|_{L^p(\Omega)}$, we have 
\[
\|\nabla\te\|_{L^p(\Omega)}\le C_p\big\|R(\tv)\big\|^{\frac{1}{p-1}}_{W^{-1,p'}(\Omega)}.
\]
On the other hand, if $\|\nabla\te\|_{L^p(\Omega)}\le \|\nabla u^*\|_{L^p(\Omega)}$, we see from Proposition \ref{lp_est} that 
\[
\|\nabla \te\|_{L^p(\Omega)}\le C_p \big\|R(\tv)\big\|_{W^{-1,p'}(\Omega)}.
\]
Therefore, by the Poincar\'e inequality and triangular inequality, we obtain
\begin{align*}
\|\te\|_{\ps} &\le C_p \left(\big\|R(\tv)\big\|_{W^{-1,p'}(\Omega)}^{\frac{1}{p-1}}+\big\|R(\tv)\big\|_{W^{-1,p'}(\Omega)}\right) \\
&\le C_p \bigg( \big\|R(v)\big\|_{W^{-1,p'}(\Omega)}^\frac{1}{p-1} + \big\|R(\tv) - R(v)\big\|_{W^{-1,p'}(\Omega)}^\frac{1}{p-1}\\
&\hspace{4mm}+ \big\|R(v)\big\|_{W^{-1,p'}(\Omega)} + \big\|R(\tv) - R(v)\big\|_{W^{-1,p'}(\Omega)}\bigg).
\end{align*}
As before, note further that
\begin{equation*}
    \big\|R(\tv)-R(v)\big\|_{W^{-1,p'}(\Omega)}
    \le \big\||\nabla\tv|^{p-2}\nabla\tv-|\nabla v|^{p-2}\nabla v\big\|_{L^{p'}(\Omega)}.
\end{equation*}
Then using Lemma \ref{alg_lem} for $1<p<2$ and \eqref{posteriori 21}, we deduce that
\[
\begin{split}
\big\|R(\tv)-R(v)\big\|_{W^{-1,p'}(\Omega)} &\le C_p \big\||\nabla\tv-\nabla v|^{p-1}\big\|_{L^{p'}(\Omega)} \\ 
&= C_p\|\nabla\tv-\nabla v\|^{p-1}_{L^p(\Omega)} \le C_p \big\|B(v)\big\|^{p-1}_{\pb}, 
\end{split}
\]
which leads us to the estimate
\begin{align*}
\|\te\|_{\ps} &\le C_p \bigg(\big\|R(v)\big\|^{\frac{1}{p-1}}_{W^{-1,p'}(\Omega)}+\big\|R(v)\big\|_{W^{-1,p'}(\Omega)} +\big\|B(v)\big\|_{\pb}+\big\|B(v)\big\|^{p-1}_{\pb}\bigg).    
\end{align*}
Therefore, under the assumptions $\big\|R(v)\big\|_{W^{-1,p'}(\Omega)}\le 1$ and $\big\|B(v)\big\|_{\pb}\le 1$, we conclude that
\begin{align*}
    \|v-u^*\|_{\ps}
    &\le \|\te\|_{\ps}+\|w\|_{\ps}\\
    &\le C_p\bigg(\big\|R(v)\big\|^{\frac{1}{p-1}}_{W^{-1,p'}(\Omega)}+\big\|R(v)\big\|_{W^{-1,p'}(\Omega)} \\    &\hspace{4mm}+\big\|B(v)\big\|_{\pb}+\big\|B(v)\big\|^{p-1}_{\pb}\bigg) \\
    &\le C_p \left( \big\|R(v)\big\|_{W^{-1,p'}(\Omega)}+\big\|B(v)\big\|^{p-1}_{\pb}\right).
\end{align*}
\end{proof}

For the purpose of error estimates for PINNs solving the $p$-Laplace equation, let us recall the definition of our loss function \eqref{def:loss} with $\delta_r=\delta_b=1$ proposed in Section \ref{sec:cons2} as follows:  $\mathcal{L}:\Theta\rightarrow\mathbb{R}$:
\[
{\mathcal{L}(\theta)=\mathcal{L}(v_{\theta})=\big\|R(v_\theta)\big\|_{W^{-1,p'}(\Omega)}^\alpha +\big\|B(v_\theta)\big\|_{\pb},}
\]
where $\alpha= \max\{1, \frac{1}{p-1} \}$ and $v_{\theta} \in \NN_\Theta$. Our first main result is the {\textit{a posteriori}} error estimate, which is a direct consequence of Proposition \ref{prop:posteriori}. This implies that, as training progresses and the loss function decreases, the error converges to zero. The detailed statement is encapsulated in the following theorem.

\begin{theorem}[{\textit{A posteriori}} estimate] \label{thm:posteriori}
    Assume that \eqref{basic setup} holds, and let $u^* \in W^{1,p}(\Omega)$ be a weak solution of \eqref{p_lap}. Then, for the neural network $v_\theta \in\NN_{\Theta}$ satisfying Assumption \ref{small assumption}, we have the following:
    \begin{itemize}
        \item[(1)] For $p>2$, there exists a constant $C_p>0$, depending continuously on $p$, such that
        \[
        \|v_\theta-u^*\|_{\ps} \le C_p \mathcal{L}(\theta)^{\frac{1}{p-1}}.
        \]
        \item[(2)] For $1<p<2$, there exists a constant $C_p>0$, depending continuously on $p$, such that
        \[
        \|v_\theta-u^*\|_{\ps} \le C_p  \mathcal{L}(\theta)^{p-1}.
        \]
    \end{itemize}
\end{theorem}

\begin{proof}
    (Case $p>2$) By Proposition \ref{prop:posteriori} and definition of the loss function $\mathcal{L}$, we have
    \[
    \|v-u^*\|_{\ps} \le C_p \left(\big\|R(v)\big\|_{W^{-1,p'}(\Omega)} +\big\|B(v)\big\|_{\pb}\right)^\frac{1}{p-1} = C_p \mathcal{L}(\theta)^\frac{1}{p-1}.
    \]

    (Case $1<p<2$) By Proposition \ref{prop:posteriori} and definition of the loss function $\mathcal{L}$, we have
    \[
    \|v-u^*\|_{\ps} \le C_p \left(\big\|R(v)\big\|_{W^{-1,p'}(\Omega)}^\frac{1}{p-1} +\big\|B(v)\big\|_{\pb}\right)^{p-1} = C_p \mathcal{L}(\theta)^{p-1}.    
    \]
\end{proof}

Next, for the second main result, we shall first prove the inequality in the opposite direction to Proposition \ref{prop:posteriori}, namely the part where the residual functional is bounded above by the error for the neural network approximation. 
\begin{proposition} \label{prop:priori}
    Assume that \eqref{basic setup} holds, and let $u^* \in W^{1,p}(\Omega)$ be a weak solution of \eqref{p_lap}. Then, for the neural network $v=v_\theta \in\NN_{\Theta}$ satisfying Assumption \ref{small assumption}, we have the following:
    \begin{itemize}
        \item[(1)] For $p > 2$, there exists a constant $C_p>0$, depending continuously on $p$, such that
    \begin{equation*}
        \big\|R(v)\big\|_{\pis} + \big\|B(v)\big\|_{\pb} \le C_p \|v-u^*\|_{\ps}.
    \end{equation*}
        \item[(2)] For $1<p<2$, there exists a constant $C_p>0$, depending continuously on $p$, such that
    \begin{equation*}
         \big\|R(v)\big\|_{\pis}^\frac{1}{p-1} + \big\|B(v)\big\|_{\pb} \le C_p \|v-u^*\|_{\ps}.
    \end{equation*}
    \end{itemize} 
\end{proposition}
\begin{proof}
For the boundary residual $\big\|B(v)\big\|_{\pb}$, by Theorem \ref{thm:trace}, we obtain
\[
\big\|B(v)\big\|_{\pb} = \|\gamma v - \gamma u^*\|_{\pb} \le C_p \| v - u^*\|_{\ps}.
\]
For the PDE residual $\big\|R(v)\big\|_{\pis}$, we see that 
\begin{equation}\label{priori 11}
\big\|R(v)\big\|_{\pis} = \big\|R(v) - R(u^*)\big\|_{\pis} \le  \big\| |\nabla v|^{p-2}\nabla v - |\nabla u^*|^{p-2} \nabla u^*\big\|_{L^{p'}(\Omega)}.
\end{equation}

(Case $p > 2$) In view of Lemma \ref{alg_lem}, an application of H\"older's inequality yields
\begin{equation*} 
\begin{aligned}
\big\| |\nabla v|^{p-2}\nabla v - |\nabla u^*|^{p-2} \nabla u^*\big\|_{L^{p'}(\Omega)}   
&\le C_p \big\|\left(|\nabla v|+|\nabla u^*|\right)^{p-2}|\nabla v-\nabla u^*| \big\|_{L^{p'}(\Omega)}\\
&\le C_p \big\| |\nabla v| + |\nabla u^*|\big\|_{L^p(\Omega)}^{p-2} \|v-u^*\|_{W^{1,p}(\Omega)} \\
&\le C_p \left( \|\nabla v - \nabla u^* \|_{L^p(\Omega)}^{p-2} + \|\nabla u^*\|_{L^p(\Omega)}^{p-2} \right) \|v - u^*\|_{\ps}.
\end{aligned}
\end{equation*}
Then by applying Proposition \ref{lp_est}, we obtain
\begin{equation}\label{priori 12}
    \big\| |\nabla v|^{p-2}\nabla v - |\nabla u^*|^{p-2} \nabla u^*\big\|_{L^{p'}(\Omega)} \le C_p \left( \|v- u^*\|_{\ps}^{p-1} + \|v-u^*\|_{\ps}\right).
\end{equation}
Inserting \eqref{priori 12} into \eqref{priori 11}, combined with the assumption $\|v- u^*\|_{W^{1,p}(\Omega)} \le 1$, leads to the desired result for the case $p > 2$. \\

(Case $1 < p < 2$) Again by Lemma \ref{alg_lem}, we have
\begin{equation} \label{priori 13}
\begin{split}
\big\| |\nabla v|^{p-2}\nabla v - |\nabla u^*|^{p-2} \nabla u^*\big\|_{L^{p'}(\Omega)} &\le C_p \big\| |\nabla v - \nabla u^*|^{p-1}\big\|_{L^{p'}(\Omega)} \\
&= C_p \|\nabla v - \nabla u^*\|_{L^p(\Omega)}^{p-1} \\
&\le C_p\|v - u^*\|_{\ps}^{p-1}.
\end{split}
\end{equation}
Inserting \eqref{priori 13} into \eqref{priori 11} completes the proof.

\end{proof}

The second main result of this section is an {\textit{a priori}} error estimate. Before proceeding further, let us introduce the notation for an optimization error. To be more specific, we shall denote the optimization error by $\delta(v_\theta) := \mathcal{L}(v_\theta) - \inf_{v_\psi \in \NN_\Theta} \mathcal{L}(v_\psi)$, which measures how close our trained neural network is to the actual minimum of the loss function. Furthermore, for simplicity as before, we assume that $\delta(v_\theta) \le 1$ in a regime where the training has sufficiently progressed. The following result concerns the {\textit{a priori}} error estimate for the PINN approximation to the $p$-Laplace equation, which can be viewed as a C\'ea-type lemma in classical numerical analysis. 

\begin{theorem}[C\'ea's lemma]  \label{thm:priori}
    Assume that \eqref{basic setup} holds, and let $u^* \in W^{1,p}(\Omega)$ be a weak solution of \eqref{p_lap}. Then, for the neural network $v_\theta \in\NN_{\Theta}$ satisfying Assumption \ref{small assumption}, we have the following:
    \begin{itemize}
    \item[(1)] For $p>2$, there exists a constant $C_p>0$, depending continuously on $p$, such that
        \[
        \|v_\theta-u^*\|_{\ps} \le C_p\left( \delta(v_\theta) + \inf_{v_\psi \in \NN_\Theta}  \|v_\psi - u^*\|_{\ps} \right)^\frac{1}{p-1}.
        \]
        \item[(2)] For $1<p<2$, there exists a constant $C_p>0$, depending continuously on $p$, such that
            \[
        \|v_\theta-u^*\|_{\ps} \le C_p \left( \delta(v_\theta) + \inf_{v_\psi \in \NN_\Theta} \|v_\psi - u^*\|_{\ps} \right)^{p-1}.
        \]
    \end{itemize}
\end{theorem}

\begin{proof}
    (Case $p>2$) By Proposition \ref{prop:posteriori}, we have 
    \[
    \|v_\theta - u^*\|_{\ps} \le C_p\mathcal{L}(v_\theta)^\frac{1}{p-1}  = C_p \left( \mathcal{L}(v_\theta) - \inf_{v_\psi \in \mathcal{N}_\Theta} \mathcal{L}(v_\psi) + \inf_{v_\psi \in \mathcal{N}_\Theta} \mathcal{L}(v_\psi) \right)^\frac{1}{p-1}.
    \]
    We apply Proposition \ref{prop:priori} to get the desired result.\\

    (Case $1<p<2$) By Proposition \ref{prop:posteriori}, we have 
    \[
    \|v_\theta - u^*\|_{\ps} \le C_p\mathcal{L}(v_\theta)^{p-1}  = C_p \left( \mathcal{L}(v_\theta) - \inf_{v_\psi \in \mathcal{N}_\Theta} \mathcal{L}(v_\psi) + \inf_{v_\psi \in \mathcal{N}_\Theta} \mathcal{L}(v_\psi) \right)^{p-1}.
    \]
    We apply Proposition \ref{prop:priori} to get the desired result.
\end{proof}

Analogously to the way error estimates are derived from C\'ea’s lemma in classical numerical methods, applying any result on the neural network approximation to the above theorem produces an explicit error estimate as a direct consequence. For example, an application of Theorem \ref{UAT_quan_1} yields the following result on the convergence rates for solutions possessing higher regularity.

\begin{corollary}[\textit{A priori} estimate] \label{conv rate}
Assume that \eqref{basic setup} holds, and suppose further that the weak 
solution $u^*$ of \eqref{p_lap} has higher regularity, namely 
$u^* \in W^{k,p}(\Omega)$ for some $k > 1$. Then, for each $n \in \mathbb{N}$, 
there exist a parameter space $\Theta_n$ of dimension $\mathcal{O}(n)$, a 
feed-forward $\tanh$ neural network $v_{\theta_n} \in \NN_{\Theta_n}$, and a 
constant $C > 0$ such that, for any arbitrarily small $\tau > 0$, there holds
\[
\|v_{\theta_n} - u^*\|_{\ps} \le \begin{cases}
\displaystyle C\left(\delta(v_{\theta_n}) 
+ \left(\frac{1}{n}\right)^{\frac{k-1-\tau}{d}} 
\|u^*\|_{W^{k,p}(\Omega)}\right)^{\frac{1}{p-1}} & \textrm{if } p > 2, \\[2ex]
\displaystyle C\left(\delta(v_{\theta_n}) 
+ \left(\frac{1}{n}\right)^{\frac{k-1-\tau}{d}} 
\|u^*\|_{W^{k,p}(\Omega)}\right)^{p-1} & \textrm{if } 1 < p < 2.
\end{cases}
\]
\end{corollary}

\begin{proof}
    It follows from Theorem \ref{thm:priori} and Theorem \ref{UAT_quan_1}.
\end{proof}

\begin{remark} \label{VPINN compute}
Aside from the theoretical analysis, evaluating $\mathcal{L}_R(v_\theta)$ and $\mathcal{L}_B(v_\theta)$ via discrete numerical computation suffers from inherent implementation challenges. First, the dual norm formulation is structurally difficult to implement in practice. Evaluating the dual norm requires computing a supremum over an infinite-dimensional space, transforming the minimization into a more complicated min-max problem. This adversarial optimization is difficult to train and prone to severe numerical instabilities (see, e.g., \cite{wgan, gan}). To circumvent this difficulty, as seen from \eqref{eqiv_vpinn}, the residual loss $\mathcal{L}_R(v_\theta)$ can be computationally replaced by the VPINN loss. More precisely, we approximate the residual loss by 
    \begin{equation} \label{VPINN loss}
    \widehat{\mathcal{L}}_R(v_\theta):=\sum_{k=1}^K \left| {\langle R(v_\theta), \phi_k \rangle_{\rm{res}}} \right|^2,
    \end{equation}
    where $\{\phi_k\}_{k=1}^\infty$ forms a basis of $W^{1,p}_0(\Omega)$ and $K \in \mathbb{N}$ is chosen to be sufficiently large. The validity of this truncated sum approximation was established in \cite{DFRPINN} for the linear elliptic case where $R(v_\theta) \in H^{-1}(\Omega)$. Specifically, as $K \to \infty$, the truncated loss $\widehat{\mathcal{L}}_R(v_\theta)$ converges to the true residual loss $\mathcal{L}_R(v_\theta)$, ensuring a closer approximation to the exact dual norm as $K$ increases when $p=2$. In practice, the empirical residual loss $\widehat{\mathcal{L}}_R(v_\theta)$ is computed by evaluating the inner products in the truncated sum using Gaussian quadrature with a sufficiently large number of quadrature points, so that we may assume that the resulting quadrature error is negligible.

    On the other hand, to compute the empirical fractional boundary loss $\widehat{\mathcal{L}}_B(v_\theta)$ to approximate the population risk $\mathcal{L}_B(v_\theta)$, we employ Monte Carlo integration with uniform random sampling on the boundary, explicitly excluding the singular points ($x=y$) to prevent division by zero.
\end{remark}

\begin{remark}
Since all the loss functions appearing in the above theorems are population 
risks, their practical applicability may seem limited. However, this issue 
can be resolved by means of the empirical risks described in the preceding 
remark. More precisely, as presented in \cite[Theorem 2]{zeinhofer2025unified}, 
the optimization error admits a decomposition suited to the practical setting: 
it can be split into a generalization error and an empirical-risk optimization 
error as
\begin{align*}
\delta(v_\theta) &= \mathcal{L}(v_\theta) - \inf_{v_\psi \in \NN_\Theta} \mathcal{L}(v_\psi) \\
&= \left(\mathcal{L}(v_\theta) - \widehat{\mathcal{L}}(v_\theta)\right) 
+ \left(\widehat{\mathcal{L}}(v_\theta) - \inf_{v_\psi \in \NN_\Theta} \widehat{\mathcal{L}}(v_\psi)\right) 
+ \left( \inf_{v_\psi \in \NN_\Theta} \widehat{\mathcal{L}}(v_\psi) 
- \inf_{v_\psi \in \NN_\Theta} \mathcal{L}(v_\psi) \right) \\
&\le {2}\sup_{\theta \in \Theta} \left|\mathcal{L}(v_\theta) 
- \widehat{\mathcal{L}}(v_\theta)\right| 
+ \left(\widehat{\mathcal{L}}(v_\theta) - \inf_{v_\psi \in \NN_\Theta} 
\widehat{\mathcal{L}}(v_\psi)\right),
\end{align*}
where $\widehat{\mathcal{L}} = \widehat{\mathcal{L}}_R + \widehat{\mathcal{L}}_B$ 
denotes the empirical risk. We emphasize that the last term on the right-hand side represents the optimization error actually incurred during training, 
rather than its theoretical counterpart defined with respect to the 
population risk. On the other hand, the first term on the right-hand side corresponds to 
the generalization error, which quantifies the discrepancy between the 
population risk and its empirical counterpart.
\end{remark}

\section{Error estimates for parametric problems}\label{sec:err_est_para}

In this section, we extend our analysis to parametric settings. As noted in the introduction, a standard PINN is inherently problem-specific: any change in the data defining the PDE, such as the source term, the boundary condition, or the nonlinear exponent, requires retraining the network from scratch, which motivates the development of a parametric framework. To address this limitation, we formulate parametric PINNs by directly embedding the problem-defining parameters into the neural network input, thereby allowing for joint training over the spatial domain $\Omega$ and the parametric space $\mathcal{P}$. This construction enables the network to approximate a solution manifold that varies with both the spatial coordinates and the parameters. 
To formalize this, we consider a parameter vector $\mu \in \mathcal{P}$ specifying the data. While a standard PINN minimizes the physics-informed loss $\mathcal{L}(\theta; \mu)$ for a fixed $\mu$, the parametric PINN aims 
to learn a global mapping over the product domain $\mathcal{P} \times \Omega$ by minimizing a population loss integrated over all admissible parameters:
\[
\mathcal{L}_{\mathrm{para}}({\theta}) := \int_{\mathcal{P}} \mathcal{L}({\theta}; {\mu}) \, \mathrm{d}\eta({\mu}),
\]
where $\eta$ is a prescribed probability measure over $\mathcal{P}$. In practice, this integration is replaced by an empirical loss function through the Monte Carlo integration. Upon completion of training, the neural network ${u}_\theta(\cdot, \cdot) : \mathcal{P} \times \Omega \to \mathbb{R}^m$ yields approximate solutions for arbitrary combinations of spatial coordinates and parameters, thus enabling near real-time solution prediction across varying parameters defining the problem settings without retraining. In terms of rapid prediction, this parametric approach can even be 
advantageous over classical numerical methods, which must solve the problem 
anew for each parameter. This advantage is particularly pronounced for 
nonlinear problems such as the present one, where classical solvers require 
iterative schemes, whereas the trained network yields the solution in a 
single forward pass. Such a parametric framework was introduced in \cite{paraPINN1, paraPINN2}; however, these works lack a rigorous error analysis. By contrast, within the DRM framework, a theoretical analysis of such parametric settings has been established. To be more specific, the authors of \cite{pPINN} established an error estimate for a parametric problem involving a variable right-hand side, exponent, and domain.

Based on these considerations, we extend our analysis of the $p$-Laplacian problem \eqref{p_lap} to the parametric setting. Specifically, by allowing the exponent $p$, the right-hand side $f$, and the boundary value $g$ to depend on a parametric variable $\mu \in \mathcal{P}$, we consider the following $p(\mu)$-Laplacian problem:
\begin{equation} \label{para p-laplacian}
\begin{aligned}
    -\operatorname{div}_x (|\nabla_x u(\mu,x)|^{p(\mu)-2} \nabla_x u(\mu, x) ) &= f(\mu,x)  \quad \text{in } \mathcal{P} \times\Omega,\\
    u(\mu,x) & = g(\mu,x) \quad \text{on } \mathcal{P} \times \partial \Omega,
\end{aligned}
\end{equation}
where $\mathcal{P} \subset \mathbb{R}^{d_\mathcal{P}}$, $d_\mathcal{P} \in \mathbb{N}$ and $\Omega \subset \mathbb{R}^{d_\Omega}$, $d_\Omega \in \mathbb{N}$ denote the parametric domain and the spatial domain, respectively. Here, $u: \mathcal{P} \times \Omega \to \mathbb{R}$ is the parametric solution associated with the parametric exponent $p : \mathcal{P} \to \mathbb{R}$, the parametric external forcing $f : \mathcal{P}\times\Omega \to \mathbb{R}$, and the parametric boundary condition $g: \calP \times \pa\Omega \to \mathbb{R}$.

Before proceeding further, we note that for each fixed $\mu \in \calP$, problem \eqref{para p-laplacian} reduces to the standard $p$-Laplace problem studied in the previous sections. In this respect, the error analysis developed in Section~\ref{sec:err_est} can be viewed as addressing the case of a single fixed parameter, and it therefore extends naturally to the parametric setting: by applying the same arguments for each $\mu$ and integrating the resulting estimates over $\calP$, we obtain the corresponding \textit{a priori} and \textit{a posteriori} error estimates. Accordingly, rather than developing fundamentally distinct theoretical tools, this section is primarily devoted to formalizing these arguments within the parametric framework. As a first step in this formalization, we ensure that the parameter-wise weak solutions exist and remain integrable over $\calP$. To this end, we impose the following assumptions on the problem data.

Let $\Omega \subset \mathbb{R}^{d_\Omega}$, $d_\Omega \in \mathbb{N}$ be a bounded Lipschitz domain and $\mathcal{P} \subset \mathbb{R}^{d_\mathcal{P}}$, $d_\mathcal{P} \in \mathbb{N}$ be a bounded measurable set. In order to make PDE data well-defined for each $\mu\in\mathcal{P}$, we assume that the forcing term $f$ and the boundary condition $g$ satisfy
\begin{equation} \label{para: basic setup}
\sup_{\mu \in \calP} \|f(\mu,\cdot)\|_{W^{-1,p'(\mu)}(\Omega)}<\infty, \quad \sup_{\mu \in \calP} \|g(\mu,\cdot)\|_{W^{1-\frac{1}{p(\mu)},p(\mu)}(\pa\Omega)} < \infty,
\end{equation}
where $p'(\mu)$ denotes the H\"older conjugate exponent of $p(\mu)$, satisfying $\frac{1}{p(\mu)} + \frac{1}{p'(\mu)} = 1$ for each $\mu \in \calP$, and $1 < p^-\le p(\mu) \le p^+ < \infty$ with $p^- := \inf_{\mu \in \calP} p(\mu)$ and $p^+ := \sup_{\mu \in \calP} p(\mu)$.

Next, we state a gradient estimate for the parametric problem, which is needed for the subsequent error analysis. As this result follows directly from Proposition \ref{lp_est}, we omit the proof.

\begin{proposition} \label{para: lp_est}
     Assume that \eqref{para: basic setup} holds, and let 
     $u(\mu, \cdot) \in W^{1,p(\mu)}(\Omega)$ be a weak solution of \eqref{para p-laplacian} corresponding to a given parameter $\mu \in \calP$. Then, for each $\mu \in \mathcal{P}$, there exists $C_p>0$, depending continuously on $p$, such that
    \[
    \|\nabla_x u(\mu,\cdot)\|_{L^{p(\mu)}(\Omega)}^{p(\mu)} \le C_p \left( \|f(\mu,\cdot)\|_{W^{-1,p'(\mu)}(\Omega)}^{p'(\mu)} + \|g(\mu,\cdot)\|^{p(\mu)}_{W^{1-\frac{1}{p(\mu)},p(\mu)}(\pa\Omega)}\right).
    \]
\end{proposition}

Before proceeding to our main error analysis for the parametric problem, given a parameter $\mu \in \mathcal{P}$, we define the PDE residual $R_\mu(v)$ and the boundary residual $B_\mu(v)$ for $v \in W^{1,p(\mu)}(\Omega)$ as
\begin{equation}\label{para_PDE_res}
    R_\mu(v):=-{\rm{div}}_x \, (|\nabla_x v(\mu, \cdot)|^{p(\mu)-2}\nabla_x v(\mu, \cdot))-f(\mu, \cdot) \in W^{-1,p'(\mu)}(\Omega)
\end{equation}
and the boundary residual
\begin{equation}\label{para_bdry_res}
    B_\mu(v ):=\gamma_\mu v(\mu, \cdot) -g(\mu, \cdot) \in W^{1-\frac{1}{p(\mu)},p(\mu)}(\pa \Omega),
\end{equation}
where $\gamma_\mu$ denotes the trace operator on $W^{1,p(\mu)}(\Omega)$. Also, $\NN_\Theta$ denotes the neural network ansatz class parameterized by $\theta \in \Theta$, consisting of functions $v_\theta(\mu, x)$ defined on $\calP \times \Omega$. Analogous to the approach in the previous section, we assume that the training process has sufficiently progressed so that the neural networks are in a near-convergence regime. Under this optimization state, we postulate the following smallness assumptions.

\begin{assumption} \label{para_small assumption}
    The neural network $v= v_\theta \in \NN_\Theta$ satisfies the following bounds: for each $\mu \in \mathcal{P}$,
    \[
    \big\|R_\mu(v)\big\|_{W^{-1,p'(\mu)}(\Omega)} \le 1, \quad \big\|B_\mu(v)\big\|_{W^{1-\frac{1}{p(\mu)},p(\mu)}(\pa\Omega)}\le 1, \quad \|v(\mu,\cdot)-u^*(\mu,\cdot)\|_{W^{1,p(\mu)}(\Omega)} \le 1.
    \]
\end{assumption}

Similar to the previous section, the following proposition provides that the neural network approximation error for each fixed parameter $\mu \in \calP$ is bounded above by the corresponding PDE and boundary residuals.

\begin{proposition} \label{prop:para_posteriori}
     Assume that \eqref{para: basic setup} holds, and let 
     $u^*(\mu, \cdot) \in W^{1,p(\mu)}(\Omega)$ be a weak solution of \eqref{para p-laplacian} corresponding to a given parameter $\mu \in \calP$. Then, for the neural network $v = v_\theta \in \mathcal{N}_\Theta$ satisfying Assumption \ref{para_small assumption}, we have the following: 
    \begin{itemize}
        \item [(1)] {If $p^->2$}, then for each $\mu \in \mathcal{P}$, there exists $C_p>0$, depending continuously on $p$, such that
    \[
    \| v(\mu, \cdot) - u^*(\mu,\cdot)\|_{W^{1,p(\mu)}(\Omega)} \le C_p \left( \big\| R_\mu(v)\big\|_{W^{-1,p'(\mu)}(\Omega)}^\frac{1}{p(\mu)-1}  + \big\|B_\mu(v)\big\|_{W^{1-\frac{1}{p(\mu)},p(\mu)}(\pa \Omega)}^\frac{1}{p(\mu)-1}\right).
    \]    
        \item [(2)] {On the other hand, if $p^+<2$}, then for each $\mu \in \mathcal{P}$, there exists $C_p>0$, depending continuously on $p$, such that
    \[
    \|v(\mu, \cdot)-u^*(\mu, \cdot)\|_{W^{1,p(\mu)}(\Omega)}
    \le C_p \left(\big\|R_\mu(v)\big\|_{W^{-1,p'(\mu)}(\Omega)}+\big\|B_\mu(v)\big\|_{W^{1-\frac{1}{p(\mu)},p(\mu)}(\pa \Omega)}^{p(\mu)-1}\right).
    \]
    \end{itemize}
\end{proposition}

\begin{proof}
The result follows directly from the arguments in the proof of Proposition \ref{prop:posteriori} together with Proposition \ref{para: lp_est}.

\end{proof}

We are now ready to present the first main result of this section: an \textit{a posteriori} error estimate for the parametric problem. In line with the discussion in Section~\ref{sec:err_est} and earlier part of this section, we define the loss function $\mathcal{L}: \Theta \to \mathbb{R}$ associated with the error estimates for $v_\theta \in \NN_\Theta$ by
\[
\calL(\theta) = \calL(v_\theta) := \int_\calP \left(\big\|R_\mu(v_\theta)\big\|_{W^{-1,p'(\mu)}(\Omega)}^{\alpha(\mu)} + \big\|B_\mu(v_\theta)\big\|_{W^{1-\frac{1}{p(\mu)},p(\mu)}(\pa \Omega)}\right) \dmu,
\]
where $\alpha(\mu)={\max\{1, \frac{1}{p(\mu)-1}\}}$. Here, we note that since the constant $C_p$ in Proposition \ref{prop:para_posteriori} depends continuously on the exponent $p$ and $p^-\leq p(\mu)\leq p^+$ for all $\mu\in\mathcal{P}$, $C_p$ is uniformly bounded above by a constant $C$ that is independent of $p$.

\begin{theorem} [\textit{A posteriori} estimate]
    Assume that \eqref{para: basic setup} holds, and let 
     $u^*(\mu, \cdot) \in W^{1,p(\mu)}(\Omega)$ be a weak solution of \eqref{para p-laplacian} corresponding to a given parameter $\mu \in \calP$. Then, for the neural network $v_\theta \in \mathcal{N}_\Theta$ satisfying Assumption \ref{para_small assumption}, we have the following: 
    \begin{itemize}
        \item[(1)] For $p^->2$, there exists a constant $C>0$ such that
        \begin{equation*}
            \int_\mathcal{P} \|v_\theta(\mu, \cdot)-u^*(\mu, \cdot)\|_{W^{1,p(\mu)}(\Omega)} \dmu \le C \calL(\theta)^\frac{1}{p^+-1}.
        \end{equation*}
        \item[(2)] For $p^+<2$, there exists a constant $C>0$ such that
        \begin{equation*}
            \int_\mathcal{P} \|v_\theta(\mu, \cdot)-u^*(\mu, \cdot)\|_{W^{1,p(\mu)}(\Omega)} \dmu \le  C \calL(\theta)^{p^--1}.
        \end{equation*}
    \end{itemize}
\end{theorem}

\begin{proof}
    (Case $p^->2$) By Proposition~\ref{prop:para_posteriori} and the inequality $\frac{1}{p^+-1} \le \frac{1}{p(\mu)-1}$ with the assumptions $\big\|R_\mu(v_\theta)\big\|_{W^{-1,p'(\mu)}(\Omega)}\le 1$ and $\big\|B_\mu(v_\theta)\big\|_{W^{1-\frac{1}{p(\mu)},p(\mu)}(\pa \Omega)} \le 1$, it follows that
    \[
    \begin{split}
    \int_\mathcal{P} \|v_\theta(\mu,\cdot)-u^*(\mu,\cdot)\|_{W^{1,p(\mu)}(\Omega)} \dmu &\le 
    \int_\mathcal{P} \bigg(\big\|R_\mu(v_\theta)\big\|_{W^{-1,p'(\mu)}(\Omega)}^\frac{1}{p(\mu)-1} + \big\|B_\mu(v_\theta)\big\|_{W^{1-\frac{1}{p(\mu)},p(\mu)}(\pa \Omega)}^\frac{1}{p(\mu)-1}\bigg) \dmu \\
    &\le C \int_\mathcal{P} \bigg(\big\|R_\mu(v_\theta)\big\|_{W^{-1,p'(\mu)}(\Omega)}^{\frac{1}{p^+-1}} + \big\|B_\mu(v_\theta)\big\|_{W^{1-\frac{1}{p(\mu)},p(\mu)}(\pa \Omega)}^{\frac{1}{p^+-1}} \bigg)\dmu.
    \end{split}
    \]
    Applying H\"older's inequality yields the desired estimate.\\
    

    (Case $p^+<2$) By Proposition \ref{prop:para_posteriori} and the inequality $p^--1 \le p(\mu)-1$ with the assumptions $\big\|R_\mu(v_\theta)\big\|_{W^{-1,p'(\mu)}(\Omega)}\le 1$ and $\big\|B_\mu(v_\theta)\big\|_{W^{1-\frac{1}{p(\mu)},p(\mu)}(\pa \Omega)} \le 1$, it follows that
    \[
    \begin{split}
    \int_\mathcal{P} \|v_\theta(\mu,\cdot)-u^*(\mu,\cdot)\|_{W^{1,p(\mu)}(\Omega)} \dmu &\le C\int_\mathcal{P} \bigg(\big\|R_\mu(v_\theta)\big\|_{W^{-1,p'(\mu)}(\Omega)} + \big\|B_\mu(v_\theta)\big\|_{W^{1-\frac{1}{p(\mu)},p(\mu)}(\pa \Omega)}^{p(\mu)-1} \bigg)\dmu \\
    &\le C \int_\calP \bigg(\big\|R_\mu(v_\theta)\big\|_{W^{-1,p'(\mu)}(\Omega)}^\frac{p^--1}{p(\mu)-1} + \big\|B_\mu(v_\theta)\big\|_{W^{1-\frac{1}{p(\mu)},p(\mu)}(\pa \Omega)}^{p^--1} \bigg)\dmu.
    \end{split}
    \]
    An application of H\"older's inequality completes the proof.
\end{proof}

In addition, we establish the following proposition, which provides the reverse inequality to Proposition~\ref{prop:para_posteriori}, serving to derive a C\'ea-type lemma as before.

\begin{proposition} \label{prop:para_priori}
    Assume that \eqref{para: basic setup} holds, and let 
     $u^*(\mu, \cdot) \in W^{1,p(\mu)}(\Omega)$ be a weak solution of \eqref{para p-laplacian} corresponding to a given parameter $\mu \in \calP$. Then, for the neural network $v = v_\theta \in \mathcal{N}_\Theta$ satisfying Assumption \ref{para_small assumption}, we have the following:
    \begin{itemize}
        \item [(1)] {If $p^->2$}, then for each $\mu \in \mathcal{P}$, there exists $C_p>0$, depending continuously on $p$, such that
        \begin{align*}
        &\big\|R_\mu(v)\big\|_{W^{-1,p'(\mu)}(\Omega)} + \big\|B_\mu(v)\big\|_{W^{1-\frac{1}{p(\mu)},p(\mu)}(\pa \Omega)}\le C_p  \|v(\mu, \cdot)-u^*(\mu, \cdot)\|_{W^{1,p(\mu)}(\Omega)} .
    \end{align*}    
        \item [(2)] {On the other hand, if $p^+<2$}, then for each $\mu \in \mathcal{P}$, there exists $C_p>0$, depending continuously on $p$, such that
    \begin{equation*}
        \big\|R_\mu(v)\big\|_{W^{-1,p'(\mu)}(\Omega)} ^\frac{1}{p(\mu)-1} + \big\|B_\mu(v)\big\|_{W^{1-\frac{1}{p(\mu)},p(\mu)}(\pa \Omega)} \le C_p  \|v(\mu, \cdot)-u^*(\mu, \cdot)\|_{W^{1,p(\mu)}(\Omega)} .
    \end{equation*}
    \end{itemize}

\end{proposition}

\begin{proof}
The result follows directly from the arguments in the proof of Proposition \ref{prop:priori} together with Proposition \ref{para: lp_est}.
\end{proof}

Building upon the preceding estimates, we now derive the second main result of this section: a C\'ea-type estimate for the parametric problem. To this end, we introduce the optimization error $\delta(v_\theta):= \mathcal{L}(v_\theta) - \inf_{v_\psi \in \mathcal{N}_\Theta} \mathcal{L}(v_\psi)$, which is assumed to satisfy $\delta(v_\theta) \le 1$ inside the convergence regime, analogous to our analysis in the previous section. A C\'ea-type lemma for the parametric problem is formulated as follows.

\begin{theorem} [C\'ea's lemma] \label{thm:para_priori}
     Assume that \eqref{para: basic setup} holds, and let 
     $u^*(\mu, \cdot) \in W^{1,p(\mu)}(\Omega)$ be a weak solution of \eqref{para p-laplacian} corresponding to a given parameter $\mu \in \calP$. Then, for the neural network $v_\theta \in \mathcal{N}_\Theta$ satisfying Assumption \ref{para_small assumption}, we have the following: 
     \begin{itemize}
        \item[(1)] For $p^->2$, there exists a constant $C>0$ such that
        \begin{equation*}
        \begin{split}
            \int_\mathcal{P} \|v_\theta(\mu,\cdot)&-u^*(\mu,\cdot)\|_{W^{1,p(\mu)}(\Omega)} \dmu \\
            &\le C \left(\delta(v_\theta) + \inf_{v_\psi \in \mathcal{N}_\Theta} \int_\calP \| v_\psi(\mu, \cdot) - u^*(\mu, \cdot) \|_{W^{1,p(\mu)}(\Omega)}\dmu\right)^\frac{1}{p^+-1}.
        \end{split}
        \end{equation*}
        \item[(2)] For $p^+<2$, there exists a constant $C>0$ such that
        \begin{equation*}
        \begin{split}
             \int_\mathcal{P} \|v_\theta(\mu,\cdot)&-u^*(\mu,\cdot)\|_{W^{1,p(\mu)}(\Omega)} \dmu \\ 
             & \le C  \left(\delta(v_\theta) + \inf_{v_\psi \in \mathcal{N}_\Theta} \int_\calP \| v_\psi(\mu, \cdot) - u^*(\mu, \cdot) \|_{W^{1,p(\mu)}(\Omega)}\dmu\right)^{p^--1}.
        \end{split}
        \end{equation*}
    \end{itemize}
\end{theorem}

\begin{proof}
    The result follows directly from the Proposition \ref{prop:para_posteriori} together with Proposition \ref{prop:para_priori}.

\end{proof}

Equipped with the C\'ea-type estimate, we deduce an \textit{a priori} error estimate together with the neural network approximation theorem (Theorem~\ref{UAT_quan_1}), which establishes the convergence rate for solutions possessing higher regularity.

\begin{corollary}[\textit{A priori} estimate]
    Assume that \eqref{para: basic setup} holds and let $u^*(\mu,\cdot) \in W^{1,p(\mu)}(\Omega)$ be a weak solution of \eqref{para p-laplacian} corresponding to a given parameter $\mu \in \calP$. Suppose further that a weak solution $u^*$ has higher regularity, namely $u^* \in W^{k,p^+}(\calP \times \Omega)$ for some $k>1$. Then, for each $n \in \mathbb{N}$, there exist a parameter space $\Theta_n$ of dimension $\mathcal{O}(n)$, a feed-forward $\tanh$ neural network $v_{\theta_n} \in \NN_{\Theta_n}$, and a constant $C>0$, such that for arbitrarily small $\tau>0$, there holds
    \[
    \begin{split}
    \int_\mathcal{P} \|v_{\theta_n}(\mu, \cdot)&-u^*(\mu, \cdot)\|_{W^{1,p(\mu)}(\Omega)} \dmu  \\
    & \le \begin{cases}
    \displaystyle C\left(\delta(v_{\theta_n}) + \left(\frac{1}{n} \right)^{\frac{k-1-\tau}{d_\calP +d_\Omega}} \|u^*\|_{W^{k,p^+}(\calP \times \Omega)}\right)^\frac{1}{p^+-1} & \text{if } 2<p^-, \\
    \displaystyle C\left(\delta(v_{\theta_n}) + \left(\frac{1}{n} \right)^{\frac{k-1-\tau}{d_\calP+d_\Omega}} \|u^*\|_{W^{k,p^+}(\calP \times \Omega)}\right)^{p^--1} & \text{if } p^+<2.
    \end{cases}
    \end{split}
    \]
\end{corollary}

\begin{proof}
    First, H\"older's inequality and Theorem \ref{UAT_quan_1} imply that
    \[
    \begin{split}
    \int_\calP \|v_\psi(\mu,\cdot)-u^*(\mu,\cdot)\|_{W^{1,p(\mu)}(\Omega)}\dmu &\le C \int_\calP\|v_\psi(\mu,\cdot)-u^*(\mu,\cdot)\|_{W^{1,p^+}(\Omega)}\dmu \\
    &\le C \left(\int_\calP\|v_\psi(\mu,\cdot)-u^*(\mu,\cdot)\|_{W^{1,p^+}(\Omega)}^{p^+}\dmu\right)^\frac{1}{p^+} \\
    &\le C \| v_\psi - u^*\|_{W^{1,p^+}(\calP \times \Omega)} \\
    & \le C \left(\frac{1}{n}\right)^\frac{k-1-\tau}{d_\calP + d_\Omega} \|u^*\|_{W^{k,p^+}(\calP \times \Omega)}.
    \end{split}
    \]
    Combining this inequality with the estimate in Theorem \ref{thm:para_priori} yields the desired result.
\end{proof}


\begin{remark}
To evaluate the parametric loss function numerically, we compute the residual norms $\big\|R_\mu(v_\theta)\big\|_{W^{-1,p'(\mu)}(\Omega)}$ and $\big\|B_\mu(v_\theta)\big\|_{W^{1-\frac{1}{p(\mu)},p(\mu)}(\partial \Omega)}$ in the same manner as in Section \ref{sec:err_est} (see Remark \ref{VPINN compute}). Furthermore, the integration with respect to the parametric variables $\mu\in\mathcal{P}$ is evaluated via Monte Carlo integration with uniform random sampling.
\end{remark}

\section{Numerical experiments}\label{sec:num_exp}
In this section, we present numerical experiments that support the theoretical analysis developed in the preceding sections. The experiments consist of three parts. First, we compare the proposed $W^{-1,p'}(\Omega)$ dual norm formulation with the standard $L^2$ residual loss, demonstrating 
the necessity and stability of the former. Second, we examine the effect of imposing the boundary condition in the fractional Sobolev trace norm rather than in the $L^2$ norm. Finally, we provide the experimental results for the proposed framework in parametric settings.

A central ingredient of the proposed framework is the dual norm formulation of the PDE residual. However, the direct evaluation of the dual norm is computationally intractable, since 
it involves a supremum over an infinite-dimensional space. For this reason, as discussed in Remark~\ref{VPINN compute}, we employ the VPINN loss \eqref{VPINN loss} to 
evaluate the residual components throughout the experiments. Crucial to this implementation is the 
choice of the basis functions 
$\{\phi_k\}_{k=1}^\infty \subset W^{1,p}_0(\Omega)$. On the computational 
domain $\Omega = (-1,1)$, we employ one of the following two families of 
basis functions:
\begin{itemize}
\item \textbf{Legendre-type basis functions:}
\begin{equation} \label{Legendre type basis}
\phi_k(x) = P_{k+2}(x) - P_k(x),
\end{equation}
where $P_k(x)$ denotes the $k$-th Legendre polynomial. This construction 
ensures that each basis function vanishes on $\partial\Omega$, owing to 
the property $P_{k+2}(\pm 1) = P_k(\pm 1)$.
\item \textbf{Sine basis functions:}
\begin{equation} \label{Sine basis}
\phi_k(x) = \sin\left(k\pi\, \frac{x+1}{2}\right),
\end{equation}
which likewise satisfy the homogeneous boundary condition 
$\phi_k(\pm 1) = 0$.
\end{itemize}

\begin{figure}[t] 
    \centering
    \includegraphics[width=1\textwidth]{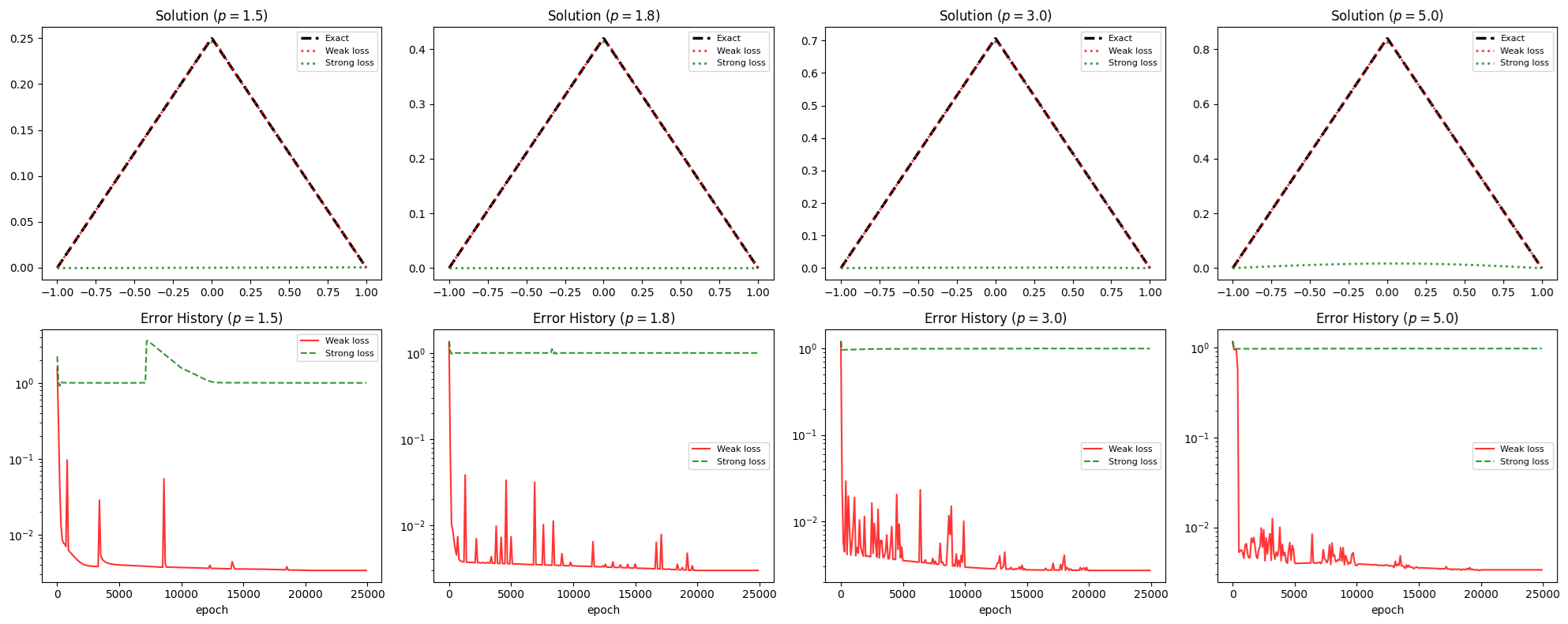}
    \caption{Comparison of weak and strong loss formulations for Example~\ref{ex:dirac}. Top row: exact (black), weak loss-based (red), and strong loss-based (green) solutions for various $p$. Bottom row: corresponding relative $L^2$ error histories during training.}
    \label{weak vs strong (dirac)}
\end{figure}

\begin{table}[t]
\centering
\small
\begin{minipage}[t]{0.495\textwidth}
\centering
\setlength{\tabcolsep}{3.5pt}
\begin{tabular}{ccc}
\toprule
$p$ & Weak error (Average) & Strong error (Average) \\
\midrule
2.05 & $2.8127 \times 10^{-3}$ & $1.0000 \times 10^{0}$ \\
3.00 & $3.0004 \times 10^{-3}$ & $1.0011 \times 10^{0}$ \\
4.00 & $3.3814 \times 10^{-3}$ & $9.9706 \times 10^{-1}$ \\
5.00 & $4.8305 \times 10^{-3}$ & $9.9374 \times 10^{-1}$ \\
\bottomrule
\end{tabular}
\end{minipage}
\hfill
\begin{minipage}[t]{0.495\textwidth}
\centering
\setlength{\tabcolsep}{3.5pt}
\begin{tabular}{ccc}
\toprule
$p$ & Weak error (Average) & Strong error (Average) \\
\midrule
1.95 & $2.8461 \times 10^{-3}$ & $1.0000 \times 10^{0}$ \\
1.80 & $3.0649 \times 10^{-3}$ & $1.0000 \times 10^{0}$ \\
1.65 & $3.3446 \times 10^{-3}$ & $1.0000 \times 10^{0}$ \\
1.50 & $3.9188 \times 10^{-3}$ & $1.0000 \times 10^{0}$ \\
\bottomrule
\end{tabular}
\end{minipage}
\caption{Average relative $L^2$ errors of the weak and strong loss formulations for Example~\ref{ex:dirac}, averaged over five independent trials, for $p > 2$ (left) and $p < 2$ (right).}
\label{tab:Dirac}
\end{table}

\subsection{Numerical validation of dual norm}

To validate the theoretical motivations of the $W^{-1,p'}(\Omega)$ dual norm formulation introduced in Sections~\ref{sec:cons} and \ref{sec:err_est}, we conduct numerical experiments comparing the proposed weak residual loss against the standard strong $L^2$ residual loss. Specifically, we consider two problems introduced in Section~\ref{sec:cons}, representing distinct analytical challenges. First, we examine a singular source with a Dirac measure (Example~\ref{ex:dirac}), where $f = \delta_0 \notin L^2(\Omega)$. This example demonstrates the necessity of the weak formulation when strong residual losses are ill-defined. To compute the strong loss for comparison, we set $f=0$, which is justified by the fact that $f=0$ a.e. on $\Omega$ and a randomly sampled point hits $x=0$ with probability zero. Second, we consider a regular source problem (Example~\ref{setup:counter_example}) with $f = -1$. Although the source is smooth, strong $L^2$ residual minimization remains unstable. Specifically, driving the strong residual to zero forces the curvature to blow up ($\|v''\|_{L^\infty} \to \infty$) for $p>2$ as characterized by Theorem~\ref{thm:counter}, while inducing numerical instability for $1<p<2$.

\paragraph{Implementation Details.} Both experiments employ an identical network architecture and optimization scheme. The solution $u(x)$ is approximated by a feed-forward neural network with $3$ hidden layers, $50$ neurons per layer, and $\tanh$ activation functions. Training is conducted in two sequential phases: an initial $20,000$ iterations using the Adam optimizer with a learning rate of $0.001$, halved every $5,000$ iterations via a step decay scheduler, followed by $5,000$ fine-tuning iterations using the L-BFGS optimizer. For the weak loss formulation \eqref{VPINN loss}, we employ $20$ Legendre-type basis functions \eqref{Legendre type basis}. Integrals in the weak loss are computed via Gaussian quadrature by dividing the domain into $10$ uniform subintervals with $10$ quadrature points per subinterval.

The qualitative solution profiles and error histories of the first and second experiments for $p=1.5, 1.8, 3, 5$ are illustrated in Figures~\ref{weak vs strong (dirac)} and \ref{weak vs strong}, respectively. Furthermore, quantitative comparisons of average relative $L^2$ errors over five independent trials, for both $p > 2$ ($p = 2.05, 3, 4, 5$) and $p < 2$ ($p = 1.95, 1.8, 1.65, 1.5$), are summarized in Table~\ref{tab:Dirac} for the first experiment and Table~\ref{tab:ex2} for the second experiment, respectively.

\begin{figure}[t] 
    \centering
    \includegraphics[width=1\textwidth]{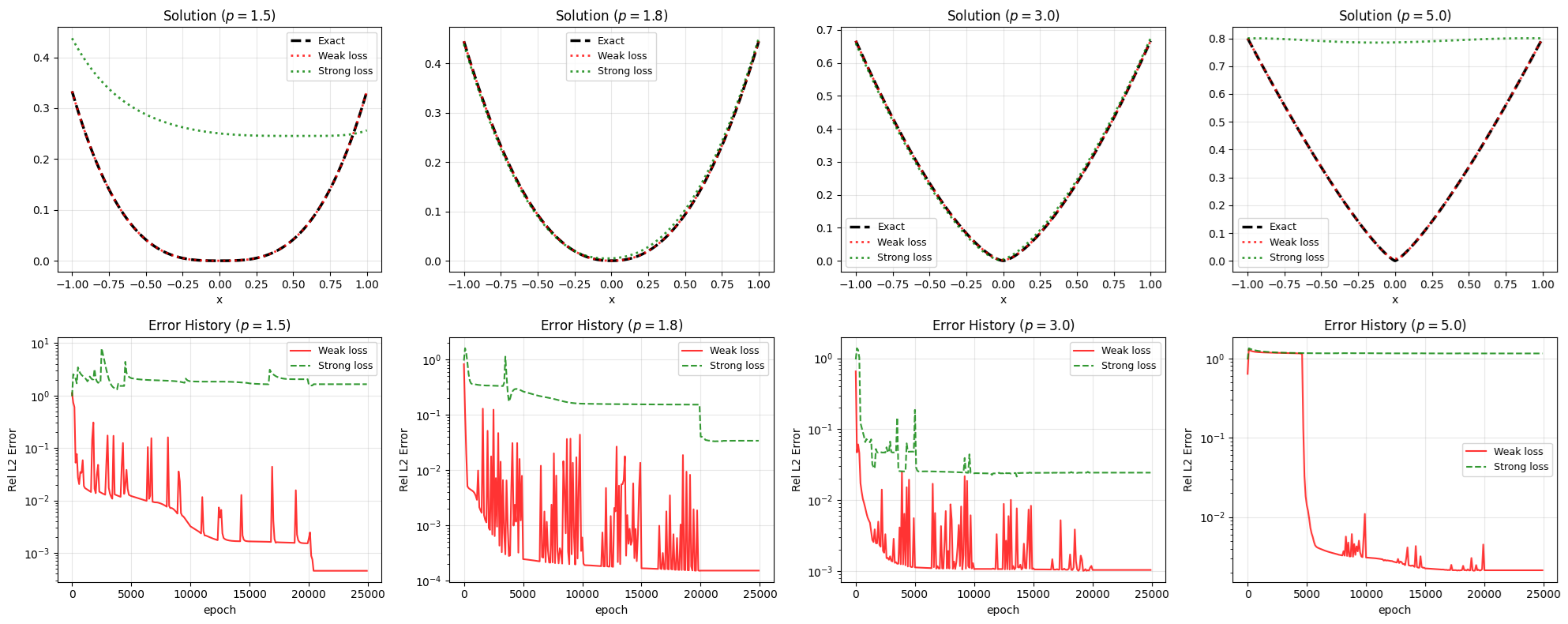}
    \caption{Comparison of weak and strong loss formulations for Example~\ref{setup:counter_example}. Top row: exact (black), weak loss-based (red), and strong loss-based (green) solutions for various $p$. Bottom row: corresponding relative $L^2$ error histories during training.}
    \label{weak vs strong}
\end{figure}

\begin{table}[t]
\centering
\small
\begin{minipage}[t]{0.495\textwidth}
\centering
\setlength{\tabcolsep}{3.5pt} 
\begin{tabular}{ccc}
\toprule
$p$ & Weak error (Average) & Strong error (Average) \\
\midrule
2.05 & $1.1288 \times 10^{-4}$ & $7.0804 \times 10^{-3}$ \\
3.00 & $9.9115 \times 10^{-4}$ & $6.0813 \times 10^{-2}$ \\
4.00 & $1.8585 \times 10^{-3}$ & $3.8912 \times 10^{-2}$ \\
5.00 & $3.1378 \times 10^{-3}$ & $1.1629 \times 10^{0}$ \\
\bottomrule
\end{tabular}
\end{minipage}
\hfill 
\begin{minipage}[t]{0.495\textwidth}
\centering
\setlength{\tabcolsep}{3.5pt}
\begin{tabular}{ccc}
\toprule
$p$ & Weak error (Average) & Strong error (Average) \\
\midrule
1.95 & $9.0174 \times 10^{-5}$ & $5.4085 \times 10^{-3}$ \\
1.80 & $1.3242 \times 10^{-4}$ & $3.7132 \times 10^{-2}$ \\
1.65 & $1.9481 \times 10^{-4}$ & $5.4740 \times 10^{-2}$ \\
1.50 & $1.9815 \times 10^{-4}$ & $3.7820 \times 10^{-1}$ \\
\bottomrule
\end{tabular}
\end{minipage}
\caption{Average relative $L^2$ errors of the weak and strong loss formulations for Example~\ref{setup:counter_example}, averaged over five independent trials, for $p > 2$ (left) and $p < 2$ (right).}
\label{tab:ex2}
\end{table}

\paragraph{Discussions.} Based on these numerical results, we make the following observations. First, regarding Example~\ref{ex:dirac}, the strong loss fails to yield any meaningful training progress, which is consistent with the theoretical fact that $f \notin L^2(\Omega)$, whereas the proposed weak loss stably achieves accurate approximations. Second, the experiments for Example~\ref{setup:counter_example} demonstrate that the weak loss formulation consistently outperforms the strong loss across all tested values of $p$; in particular, under highly nonlinear regimes (e.g., $p=1.5$ and $p=5$), the strong formulation fails to train, whereas the weak loss maintains robust performance. Finally, Theorem~\ref{thm:priori} predicts that the theoretical upper bound grows larger as $p$ increases for $p > 2$ and as $p$ decreases for $p < 2$. As shown in Tables~\ref{tab:Dirac} and \ref{tab:ex2}, the empirical approximation errors align with this prediction, validating our error bounds for both experiments.


\subsection{Numerical validation of fractional boundary loss}

To provide numerical justification for the fractional boundary norm discussed in Section~\ref{sec:err_est}, we evaluate the impact of the boundary loss formulation on model performance. Recall that for a solution $u \in W^{1,p}(\Omega)$, its natural trace resides in the fractional Sobolev space $\pb$. Based on this theoretical requirement, we consider the manufactured solution $u(x,y) = (1-x^2)(1-y^2)$ on $\Omega = (-1,1) \times (-1,1)$. To examine the effect of the boundary formulation, we compare two setups, one utilizing the fractional Sobolev trace norm $\pb$ and the other using the standard $L^2(\partial\Omega)$ norm, while maintaining a consistent $W^{-1,p'}(\Omega)$ interior loss.

\paragraph{Implementation Details.} Both loss formulations share identical interior loss evaluations and neural network configurations. The solution $u(x,y)$ is approximated by a feed-forward neural network with $3$ hidden layers, $50$ neurons per layer, and Swish activation functions. Training is conducted using the Adam optimizer for $40,000$ iterations, with an initial learning rate of $0.001$ halved every $5,000$ iterations via a step decay scheduler. The interior $W^{-1,p'}(\Omega)$ norm is evaluated via \eqref{VPINN loss} using $100$ Sine basis functions \eqref{Sine basis} ($10$ basis functions per spatial direction). Interior integrals are computed via Gaussian quadrature on an $8 \times 8$ uniform grid, employing $8 \times 8$ quadrature points per grid cell. For the boundary components, the fractional Sobolev norm is evaluated via the Monte Carlo approach based on $400$ uniform random samples along $\partial\Omega$.

The qualitative solution profiles and error histories for $p=1.5, 1.8, 3, 5$ are illustrated in Figure~\ref{fractional vs L2}. Furthermore, quantitative comparisons of average relative $L^2$ errors over five independent trials, for both $p > 2$ ($p = 2.05, 3, 4, 5$) and $p < 2$ ($p = 1.95, 1.8, 1.65, 1.5$), are summarized in Table~\ref{tab:fracvsl2}.

\begin{figure}[t]
    \centering
    \includegraphics[width=1\textwidth]{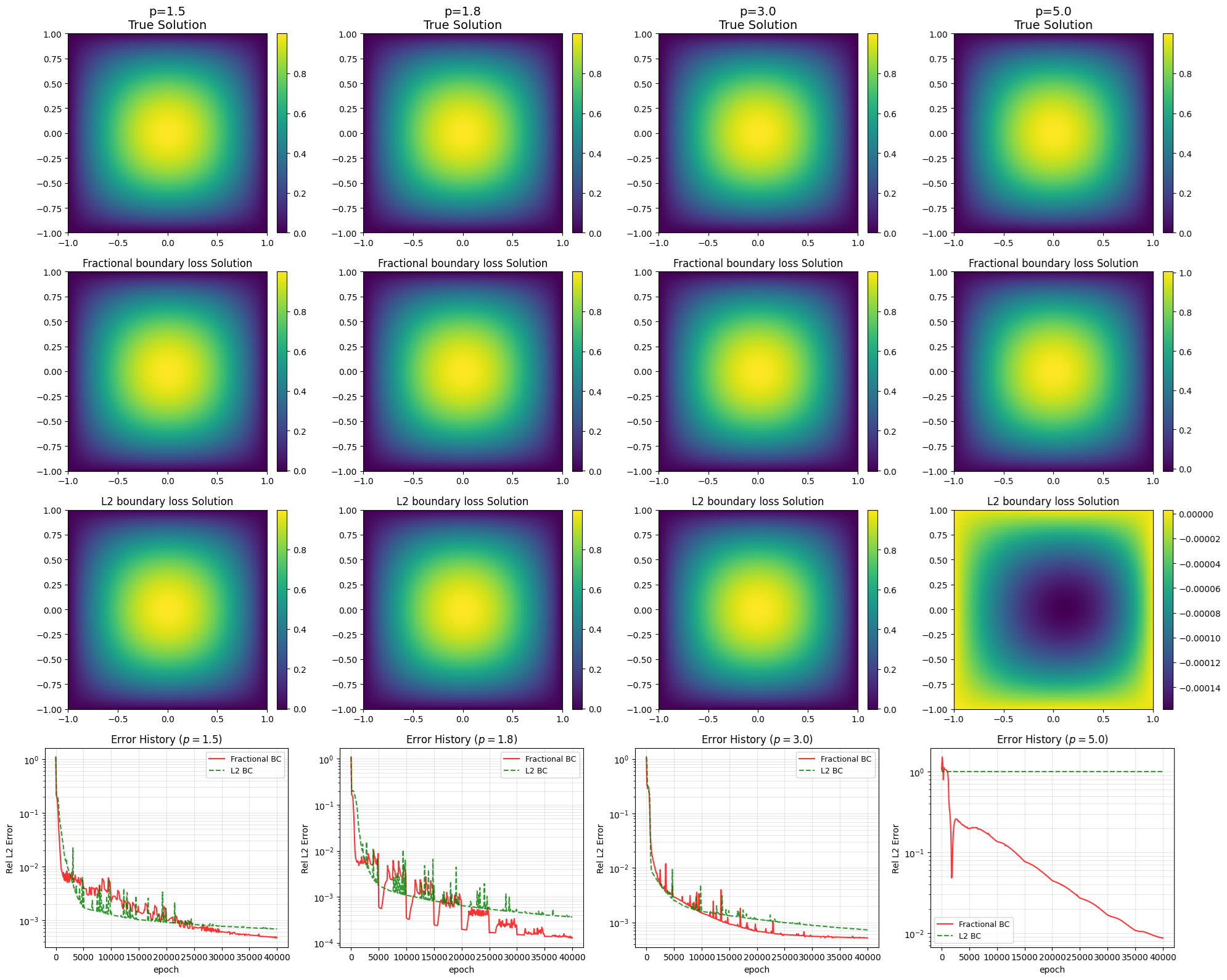}
    \caption{Comparison of models trained with the fractional and $L^2$ boundary losses for various values of $p$. From top to bottom: exact solutions, solutions obtained with the fractional boundary loss, solutions obtained with the $L^2$ boundary loss, and relative $L^2$ error histories, where the red and green lines correspond to the fractional and $L^2$ boundary losses, respectively.}
    \label{fractional vs L2}
\end{figure}

\begin{table}[t]
\centering
\small
\begin{minipage}[t]{0.495\textwidth}
\centering
\setlength{\tabcolsep}{3.5pt}
\begin{tabular}{ccc}
\toprule
$p$ & Frac error (Average) & $L^2$ error (Average) \\
\midrule
2.05 & $1.5490 \times 10^{-4}$ & $3.9831 \times 10^{-4}$ \\
3.00 & $5.7732 \times 10^{-4}$ & $5.1330 \times 10^{-4}$ \\
4.00 & $2.3464 \times 10^{-3}$ & $3.7594 \times 10^{-3}$ \\
5.00 & $3.1050 \times 10^{-2}$ & $9.9997 \times 10^{-1}$ \\
\bottomrule
\end{tabular}
\end{minipage}
\hfill
\begin{minipage}[t]{0.495\textwidth}
\centering
\setlength{\tabcolsep}{3.5pt}
\begin{tabular}{ccc}
\toprule
$p$ & Frac error (Average) & $L^2$ error (Average) \\
\midrule
1.95 & $9.0809 \times 10^{-5}$ & $2.5694 \times 10^{-4}$ \\
1.80 & $1.0237 \times 10^{-4}$ & $2.4693 \times 10^{-4}$ \\
1.65 & $1.5207 \times 10^{-4}$ & $3.8910 \times 10^{-4}$ \\
1.50 & $4.0264 \times 10^{-4}$ & $4.3689 \times 10^{-4}$ \\
\bottomrule
\end{tabular}
\end{minipage}
\caption{Average relative $L^2$ errors of the fractional and $L^2$ boundary loss formulations, averaged over five independent trials, for $p > 2$ (left) and $p < 2$ (right).}
\label{tab:fracvsl2}
\end{table}

\paragraph{Discussions.} As observed in Table~\ref{tab:fracvsl2}, employing the fractional boundary norm yields lower approximation errors across nearly all tested values of $p$. Specifically, for $p=5$, the standard $L^2$ boundary loss fails to train the model, whereas the fractional loss remains robust. This performance disparity is directly attributed to the underlying Sobolev embedding $W^{1-\frac{1}{p},p}(\partial\Omega) \hookrightarrow L^2(\partial\Omega)$ valid for $p > \frac{2d}{d+1}$ (where $d$ is the dimension of $\Omega$), which establishes that $\|\cdot \|_{L^2(\partial\Omega)} \lesssim \|\cdot\|_{\pb}$. Consequently, the fractional Sobolev norm serves as a stricter criterion that effectively controls the $L^2$ loss, although the standard $L^2$ boundary loss may still structurally suffice for moderate regimes such as $p=3$. Furthermore, we observe that the empirical errors in Table~\ref{tab:fracvsl2} align with Theorem~\ref{thm:priori}, where the theoretical upper bound grows larger as $p$ increases for $p > 2$ and as $p$ decreases for $p < 2$.

\subsection{Numerical experiments on parametric problems}

\begin{figure}[t] 
    \centering
    \includegraphics[width=1\textwidth]{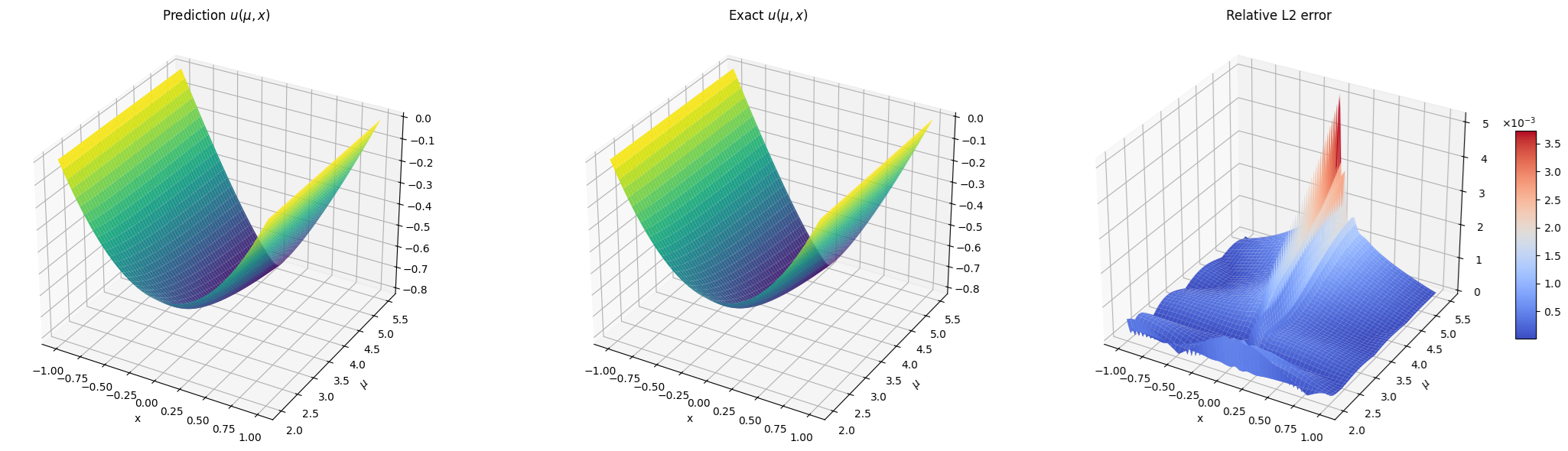}
    \caption{Comparison between the predicted and exact solutions for the parametric-exponent problem with $p(\mu)>2$. The left and center plots display the surfaces of the predicted and exact solutions for $u(\mu, x)$, respectively, while the right plot shows the relative $L^2$ error across the spatial and parameter domains.}
    \label{para_weak}
\end{figure}

\begin{figure}[t] 
    \centering
    \includegraphics[width=1\textwidth]{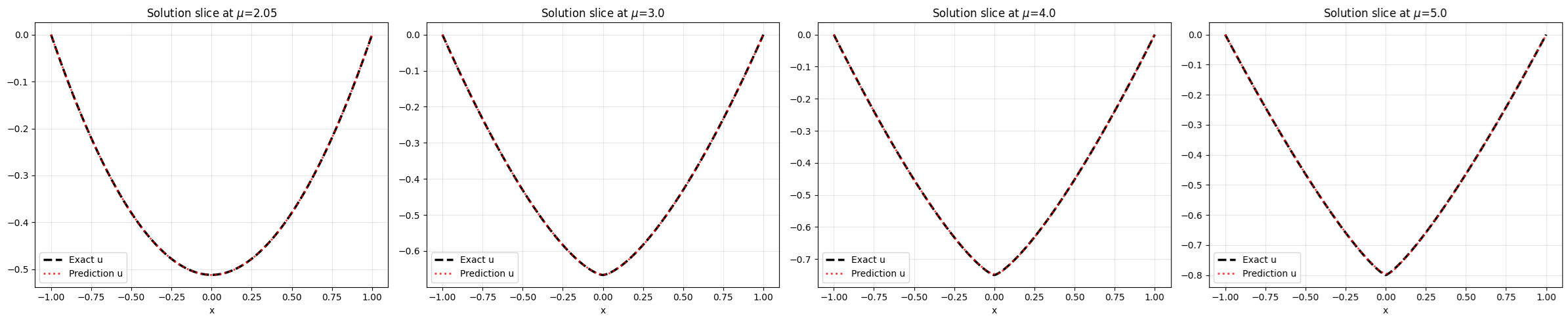}
    \caption{Comparison between the predicted and exact solutions evaluated at specific slices of the exponent $\mu \in \{2.05, 3, 4, 5\}$, where varying $\mu$ alters the operator nonlinearity $p(\mu)$ and the resulting solution profile.
    The black lines denote the exact solutions, while
    the red lines represent the predicted solutions.}
    \label{para_weak_slice}
\end{figure}

\begin{figure}[t] 
    \centering
    \includegraphics[width=1\textwidth]{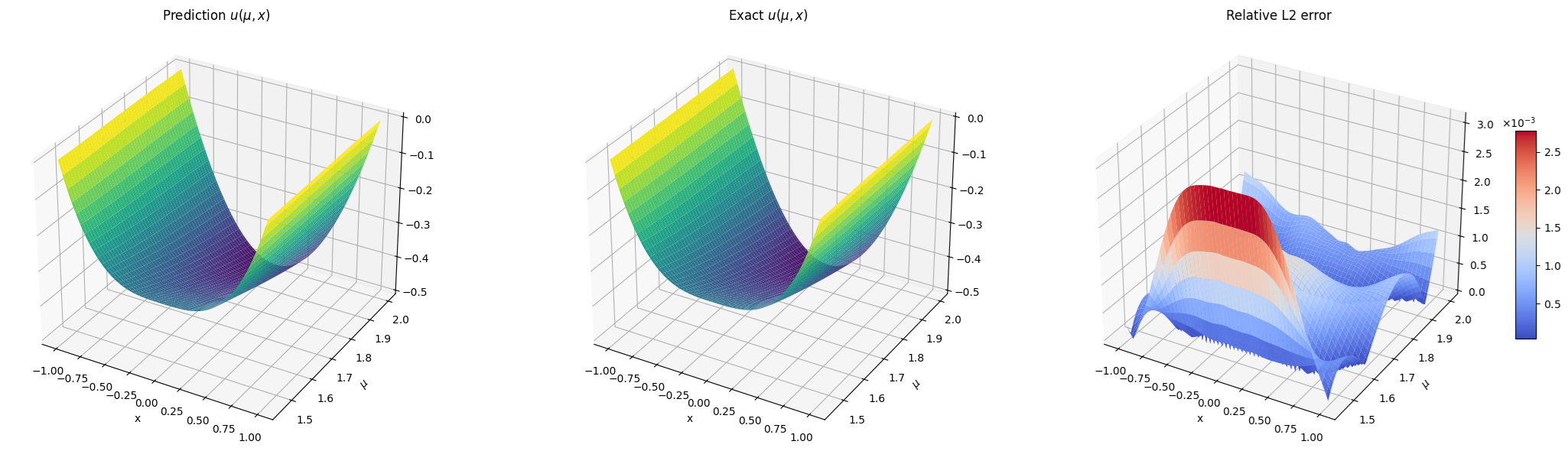}
    \caption{Comparison between the predicted and exact solutions for the parametric-exponent problem with $p(\mu)<2$. The left and center plots display the surfaces of the predicted and exact solutions for $u(\mu, x)$, respectively, while the right plot shows the relative $L^2$ error across the spatial and parameter domains.}
    \label{para_weak_p<2}
\end{figure}

\begin{figure}[t] 
    \centering
    \includegraphics[width=1\textwidth]{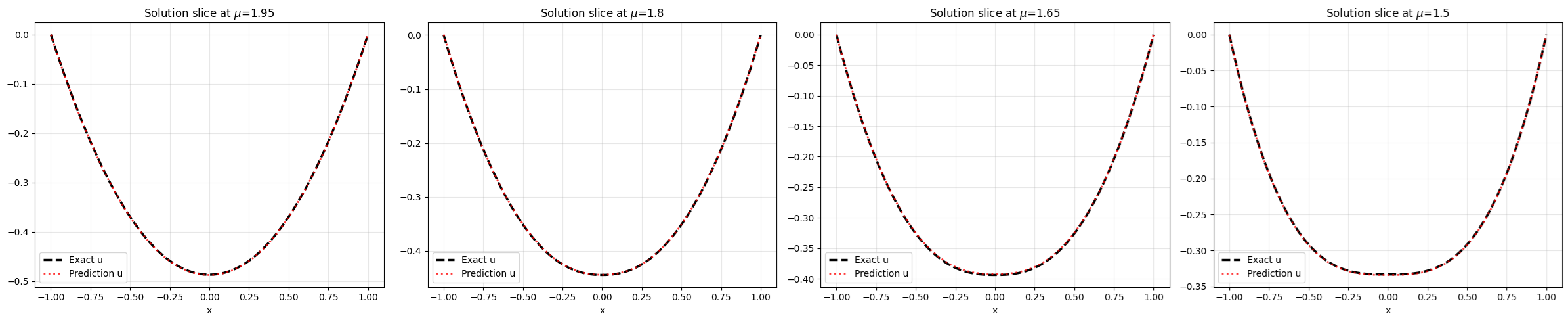}
    \caption{Comparison between the predicted and exact solutions evaluated at specific slices of the exponent $\mu \in \{1.95, 1.8, 1.65, 1.5\}$, where varying $\mu$ alters the operator nonlinearity $p(\mu)$ and the resulting solution profile. The black lines denote the exact solutions, while the red lines represent the predicted solutions.}
    \label{para_weak_slice_p<2}
\end{figure}

\begin{table}[t]
\centering
\small
\begin{minipage}[t]{0.495\textwidth}
\centering
\setlength{\tabcolsep}{3.5pt} 
\begin{tabular}{cc}
\toprule
$\mu$ & $L^2$ error (Average) \\
\midrule
2.05 & $1.2892 \times 10^{-3}$ \\
3.00 & $3.9298 \times 10^{-4}$ \\
4.00 & $7.3375 \times 10^{-4}$ \\
5.00 & $9.9399 \times 10^{-4}$ \\
\bottomrule
\end{tabular}
\end{minipage}
\hfill 
\begin{minipage}[t]{0.495\textwidth}
\centering
\setlength{\tabcolsep}{3.5pt}
\begin{tabular}{cc}
\toprule
$\mu$ & $L^2$ error (Average) \\
\midrule
1.95 & $1.5174 \times 10^{-3}$ \\
1.80 & $2.7256 \times 10^{-3}$ \\
1.65 & $3.8536 \times 10^{-3}$ \\
1.50 & $7.0549 \times 10^{-3}$ \\
\bottomrule
\end{tabular}
\end{minipage}
\caption{Average relative $L^2$ errors evaluated at specific slices of the exponent $\mu$, where the left and right tables correspond to the parametric domains $\calP_1=(2, 5.5)$ and $\calP_2=(1.45,2)$, respectively.}
\label{tab:para}
\end{table}

Next, we consider the numerical experiments for the parametric problems. We first categorize our empirical investigation into two primary setups based on the role of the parameter $\mu$:

\begin{enumerate}
    \item \textbf{Parametric exponent:} The parametric exponent $p(\mu)$ varies with $\mu$, directly altering the structural nonlinearity of the $p$-Laplacian operator.
    \item \textbf{Parametric data:} The exponent remains fixed at $p(\mu) = p$, while the external forcing term varies with the parameter $\mu$, i.e., $f = f(\mu,x)$.
\end{enumerate}

\paragraph{Case 1: Parametric exponent.} We extend the non-parametric framework of Example~\ref{setup:counter_example} to the parametric setting by defining $p(\mu) := \mu$ over the parameter domain $\calP$. Specifically, we fix the spatial domain as $\Omega = (-1, 1)$ and consider two distinct parameter domains: $\calP_1 = (2, 5.5)$ for $p(\mu) > 2$, and $\calP_2 = (1.45, 2)$ for $p(\mu) < 2$. Under these configurations, we solve the following parametric problem:
\[
-\diver_x \left(|\nabla_x u(\mu, x)|^{\mu-2} \nabla_x u(\mu, x)\right) = -1 \quad \text{in } \calP \times \Omega,
\]
subject to the boundary conditions $u(\mu, \pm 1) = 0$. The exact solution is given by:
\[
u^*(\mu,x) = \frac{\mu-1}{\mu} \left(|x|^\frac{\mu}{\mu-1}-1\right).
\]

\begin{table}[h]
\centering
\small
\begin{tabular}{ccc}
\toprule
$\mu$ & $L^2$ error at $p=1.8$ (Average) & $L^2$ error at $p=2.5$ (Average) \\
\midrule
1.5 & $5.4303 \times 10^{-4}$ & $4.8773 \times 10^{-4}$ \\
2.0 & $8.2709 \times 10^{-4}$ & $3.8844 \times 10^{-4}$ \\
2.5 & $7.2281 \times 10^{-4}$ & $3.0850 \times 10^{-4}$ \\
3.0 & $1.1338 \times 10^{-3}$ & $3.5475 \times 10^{-4}$ \\
\bottomrule
\end{tabular}
\caption{Average relative $L^2$ errors evaluated at specific parameter slices $\mu \in \{1.5, 2, 2.5, 3\}$ for the parametric external-forcing problem with fixed exponents $p=1.8$ and $p=2.5$.}
\label{tab:vrhs}
\end{table}

\begin{figure}[h] 
    \centering
    \includegraphics[width=1\textwidth]{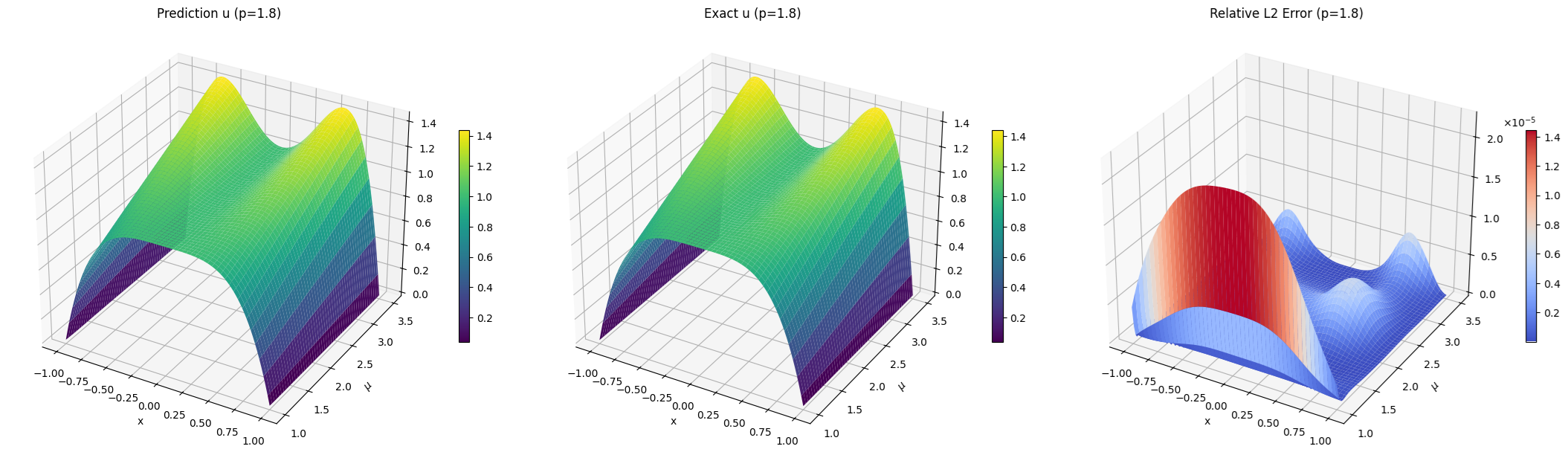}
    \caption{Comparison between the predicted and exact solutions for the parametric external-forcing problem at $p=1.8$. The left and center plots display the surfaces of the predicted and exact solutions for $u(\mu, x)$, respectively, while the right plot shows the relative $L^2$ error across the spatial and parameter domains.}
    \label{vrhs_p1.8}
\end{figure}

\begin{figure}[h] 
    \centering
    \includegraphics[width=1\textwidth]{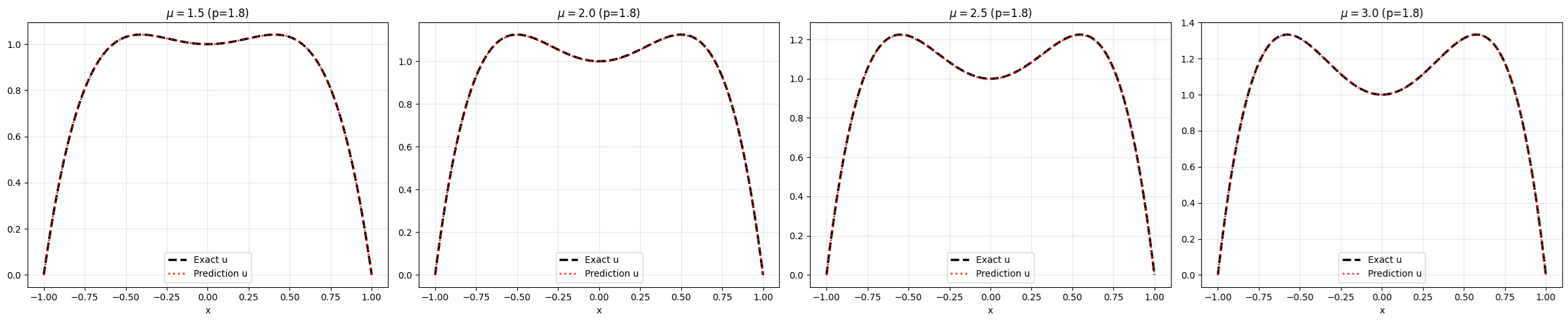}
    \caption{Comparison between the predicted and exact solutions evaluated at specific parameter slices $\mu \in \{1.5, 2, 2.5, 3\}$ for the parametric external-forcing problem with $p=1.8$, where varying $\mu$ alters the forcing term $f(\mu,x)$ and the resulting solution profile. The black lines denote the exact solutions, while
    the red lines represent the predicted solutions.}
    \label{vrhs_slice_p1.8}
\end{figure}

\begin{figure}[h] 
    \centering
    \includegraphics[width=1\textwidth]{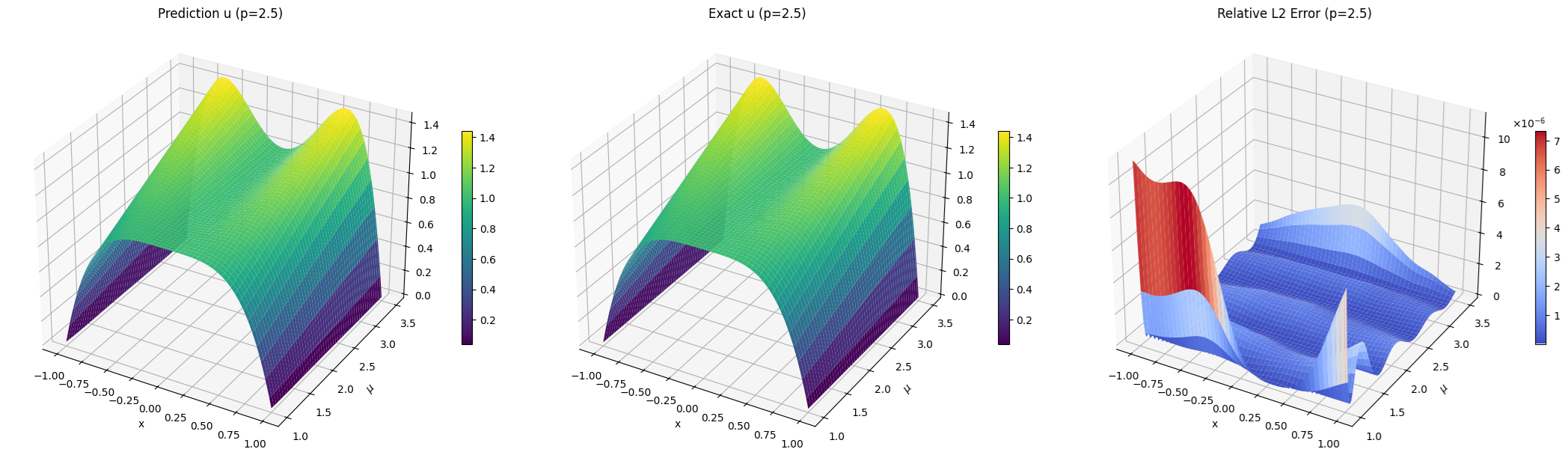}
    \caption{Comparison between the predicted and exact solutions for the parametric external-forcing problem at $p=2.5$. The left and center plots display the surfaces of the predicted and exact solutions for $u(\mu, x)$, respectively, while the right plot shows the relative $L^2$ error across the spatial and parameter domains.}
    \label{vrhs_p2.5}
\end{figure}

\begin{figure}[h] 
    \centering
    \includegraphics[width=1\textwidth]{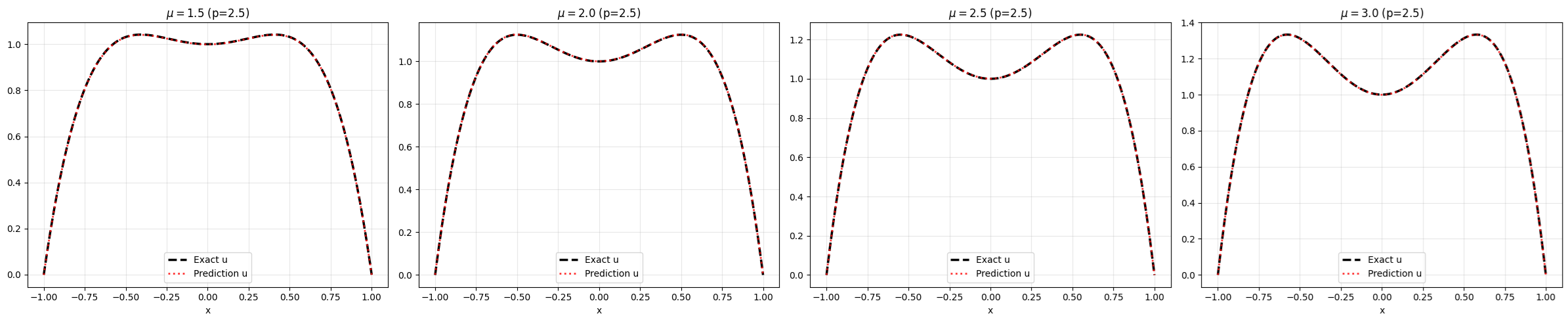}
    \caption{Comparison between the predicted and exact solutions evaluated at specific parameter slices $\mu \in \{1.5, 2, 2.5, 3\}$ for the parametric external-forcing problem with $p=2.5$, where varying $\mu$ alters the forcing term $f(\mu,x)$ and the resulting solution profile. The black lines denote the exact solutions, while
    the red lines represent the predicted solutions.}
    \label{vrhs_slice_p2.5}
\end{figure}

\paragraph{Implementation Details (Case 1).} We formulate the loss function using the $W^{-1,\mu'}(\Omega)$ norm for the spatial variable, where the dual norm is evaluated via \eqref{VPINN loss} using $20$ Sine basis functions. Spatial integration is performed via Gaussian quadrature by dividing the spatial domain $\Omega = (-1,1)$ into $10$ uniform subintervals with $10$ quadrature points per subinterval. For the parametric domain $\calP$, we adopt the Monte Carlo integration approach, sampling $100$ and $20$ random points per iteration from $\calP_1$ and $\calP_2$, respectively. The neural network architecture consists of $3$ hidden layers with $50$ neurons per layer and Swish activation functions. Training is conducted using the Adam optimizer for $15,000$ iterations with an initial learning rate of $0.001$, halved every $5,000$ iterations via a step decay scheduler, followed by $5,000$ fine-tuning iterations using L-BFGS.

The qualitative surface predictions and representative slice-wise profiles for $p(\mu) > 2$ ($\calP=\calP_1$) are depicted in Figures~\ref{para_weak} and \ref{para_weak_slice}, respectively, while those for $p(\mu) < 2$ ($\calP=\calP_2$) are shown in Figures~\ref{para_weak_p<2} and \ref{para_weak_slice_p<2}. Quantitative slice-wise relative $L^2$ errors for both parameter domains are summarized in Table~\ref{tab:para}.

\paragraph{Discussions (Case 1).} Inspecting Figures~\ref{para_weak_slice} and \ref{para_weak_slice_p<2}, together with Table~\ref{tab:para}, we observe that the trained model accurately reproduces the solution profiles across the entire range of the parameter, including the representative values reported therein. This 
demonstrates that a single network trained with the proposed weak loss can capture the varying nonlinearity of the operator over a continuous range of exponents, without the need for retraining for each individual parameter.

\paragraph{Case 2: Parametric data.} Next, we examine the parametric data problem, where the exponent remains fixed at $p(\mu) = p$, while the forcing term varies with $\mu$. We set the spatial domain as $\Omega = (-1, 1)$ and the parameter domain as $\calP = (1, 3.5)$. We consider two separate experiments by fixing the exponent at $p=1.8$ ($p<2$) and $p=2.5$ ($p>2$). In both setups, we adopt the manufactured solution $u(x, \mu) = (1 - x^2)(1 + \mu x^2)$, which satisfies homogeneous Dirichlet boundary conditions $g = 0$. The corresponding source term $f(\mu,x)$, derived via \eqref{para p-laplacian}, is given by:
\[
f(\mu,x) = (p-1) 2^{p-1} |x|^{p-2} |\mu - 1 - 2\mu x^2|^{p-2} (1 - \mu + 6\mu x^2).
\]

\paragraph{Implementation Details (Case 2).} The spatial loss is constructed using the $W^{-1,p'}(\Omega)$ norm via \eqref{VPINN loss} with $50$ Sine basis functions. Spatial integration is performed via Gaussian quadrature by partitioning $\Omega = (-1,1)$ into $20$ uniform subintervals with $20$ quadrature points per subinterval. Integration over the parameter domain $\calP = (1, 3.5)$ is handled via the Monte Carlo method with $70$ random points per iteration. The network architecture and optimization scheme are identical to Case 1, employing three hidden layers with $50$ neurons per layer and Swish activation functions. The model is trained using the Adam optimizer for $15,000$ iterations, followed by $5,000$ fine-tuning iterations using L-BFGS.

The qualitative surface predictions and representative slice-wise profiles for $p=1.8$ are depicted in Figures~\ref{vrhs_p1.8} and \ref{vrhs_slice_p1.8}, respectively, while those for $p=2.5$ are shown in Figures~\ref{vrhs_p2.5} and \ref{vrhs_slice_p2.5}. Quantitative slice-wise relative $L^2$ errors for both exponents are summarized in Table~\ref{tab:vrhs}.

\paragraph{Discussions (Case 2).} 
Similar to Case 1, Figures~\ref{vrhs_slice_p1.8} and \ref{vrhs_slice_p2.5} along with Table~\ref{tab:vrhs} confirm that the model accurately predicts the solution profiles across varying parameter slices. This demonstrates that the proposed weak formulation generalizes robustly to continuously varying forcing terms without requiring individual retraining.

\section{Conclusion}\label{sec:concl}
In this work, we have developed a robust training framework for PINNs applied to the $p$-Laplace equation, in which the PDE residual is measured 
in the dual norm, and the boundary residual in the fractional Sobolev norm. The robustness of the framework stems from its alignment with the natural 
functional setting of the problem: the dual norm formulation keeps the loss functional well-defined in low-regularity regimes where classical 
strong-form formulations may fail to apply, while the fractional Sobolev norm precisely captures the trace regularity of weak solutions. To substantiate this robustness theoretically, we have established a rigorous error analysis comprising both \textit{a priori} and \textit{a posteriori} estimates: the former guarantees the convergence of the proposed method, whereas the latter ensures that the training loss controls the true error in the appropriate Sobolev norm. Finally, we have extended the framework and its error analysis to the parametric setting, in which the data (such as the source term and the boundary condition) and the nonlinear exponent $p$ enter as additional inputs to the neural network. The resulting error bounds generalize consistently across the parameter space, so that a single trained model can accurately predict solutions for an entire family of 
$p$-Laplace problems without individual retraining.

Several directions remain for future research. First, from a computational standpoint, we plan to investigate more efficient techniques for evaluating 
the $W^{-1,p'}(\Omega)$ dual norm and the $\pb$ fractional Sobolev norm, including the optimal choice of test basis functions and the use of adversarial networks that minimize the dual-norm residual via min-max optimization. Second, we aim to extend the proposed framework to a broader class of nonlinear PDEs with analogous mathematical structure, such as generalized Newtonian fluid models, in which 
low-regularity and singular behaviors pose similar challenges.

\section*{Acknowledgements}
The authors thank Prof. Jae-Hwan Choi for the helpful discussions on this project.
Kyueon Choi is supported by the National Research Foundation of Korea Grant funded by the Korea Government (RS-2024-00336346 and RS-2024-00406821). Seungchan Ko is supported by National Research Foundation of Korea Grant funded by the Korean Government (RS-2023-00212227). Dohyun Kwon is partially supported by the National Research Foundation of Korea (NRF) grant funded by the Korea government (MSIT) (No. RS-2023-00252516, No. RS-2024-00408003,  No. RS-2026-25488663, and No. RS-2026-25613008), the POSCO Science Fellowship of POSCO TJ Park Foundation, and the Korea Institute for Advanced Study. This research was supported by the Yonsei University Research Fund of 2026-22-0251.

\bibliographystyle{abbrv}
\bibliography{references}

\end{document}